\documentclass[11pt,reqno]{amsart}
\usepackage{amsmath,amsfonts,amscd,amssymb,epsfig,euscript,bm}
\usepackage{amsxtra,mathrsfs,footnote}
\usepackage{pdfsync}
\newtheorem{theorem}{Theorem}
\newtheorem{proposition}{Proposition}
\newtheorem{lemma}{Lemma}[section]
\newtheorem{corollary}{Corollary}
\theoremstyle{definition}
\newtheorem{definition}{Definition}
\newtheorem{example}{Example}
\theoremstyle{remark} 
\newtheorem{remark}{Remark}
\numberwithin{equation}{section}
\newcommand{\field}[1]{\ensuremath{\mathbb{#1}}}
\newcommand{\CC}{\field{C}}

\newcommand{\ZZ}{\field{Z}}

\DeclareMathOperator{\Tr}{Tr}

\newcommand{\curly}[1]{\mathscr{#1}}

\newcommand{\cB}{\curly{B}}

\newcommand{\cD}{\curly{D}}

\newcommand{\cF}{\curly{F}}

\newcommand{\cW}{\curly{W}}

\newcommand{\A}{\field{A}}
\newcommand{\NN}{\field{N}}

\newcommand{\OO}{\mathcal{O}}
\newcommand{\pp}{\mathfrak{p}}
\newcommand{\g}{\mathfrak{g}}

\makeatletter
\def\blfootnote{\xdef\@thefnmark{}\@footnotetext}
\makeatother

 \DeclareMathOperator{\Hom}{Hom}
\DeclareMathOperator{\Symm}{Sym} 
\DeclareMathOperator{\Aut}{Aut}
 
 \DeclareMathOperator{\End}{End}
\DeclareMathOperator{\Res}{Res} 
\DeclareMathOperator{\Div}{Div}
\DeclareMathOperator{\Ind}{ind}
\usepackage{hyperref}
\hypersetup{colorlinks,citecolor=blue,plainpages=false,hypertexnames=false}
\begin{document}
\title[Quantum field theories on algebraic curves]
{Quantum field theories on algebraic curves I. Additive bosons}
\author{Leon A. Takhtajan}
\address{Department of Mathematics \\
Stony Brook University\\ Stony Brook, NY 11794-3651 \\ 
USA}
\address{The Euler International Mathematical Institute, St. Petersburg Department of Steklov Mathematical Institute, Pesochnaya nab. 10, St.Petersburg 197022, 
Russia }
\email{leontak@math.sunysb.edu}
\begin{abstract}
Using Serre's adelic interpretation of cohomology, we develop a `differential and integral calculus' on an algebraic curve $X$ over an algebraically closed filed $k$ of constants of characteristic zero, define algebraic analogs of additive multi-valued functions on $X$ and prove corresponding generalized residue theorem. Using the representation theory of the global Heisenberg and lattice Lie algebras, we formulate quantum field theories of additive and charged bosons on an algebraic curve $X$. These theories are naturally connected with the algebraic de Rham theorem.
We prove that an extension of global symmetries (Witten's additive Ward identities) from the $k$-vector space of rational functions on $X$ to the vector space of additive multi-valued functions uniquely determines these quantum theories of additive and charged bosons.  

Bibliography: 18 titles.
\end{abstract}
\maketitle

\section{Introduction}\blfootnote{This work was partially supported by the NSF grants DMS-0204628, DMS-0705263  and DMS-1005769.} 
The classical theory of compact Riemann surfaces has an algebraic counterpart, the theory of algebraic functions of one variable over an arbitrary field of constants, as developed by Dedekind and Weber.
The introduction of differentials into the algebraic theory by Artin and Hasse and the definition of id\`{e}les and ad\`{e}les adeles by Chevalley and Weil opened the way for application of infinite-dimensional methods to the theory of algebraic curves. Classical examples of using such methods are given by Serre's adelic interpretation of cohomology and the Riemann-Roch theorem
 \cite{Serre} and Tate's proof of the general residue theorem \cite{Tate}.    In 1987, Arbarello, de Concini and Kac \cite{Arbarello} interpreted Tate's
approach in terms of central extensions of infinite-dimensional Lie algebras and gave a new proof of Weil's celebrated reciprocity law using the infinite-wedge representation. 

In 1987, Kazhdan \cite{Kazhdan} and Witten \cite{Witten-1} proposed an adelic formulation of the quantum field theory of one-component free fermions on an algebraic curve, and Witten \cite{Witten-2} outlined an approach to other quantum field theories. Let $X$ be an algebraic curve over an algebraically closed field $k$ of constants, and let $L$ be a spin structure on $X$. We write $\mathcal{M}(L)$ for the infinite-dimensional $k$-vector space of meromorphic sections  of $L$ over $X$, and $\mathcal{M}_{P}$ for the completions of $\mathcal{M}(L)$ at all points $P\in X$. In outline, the approach of \cite{Kazhdan, Witten-1} can be described in terms of the following objects.
\begin{itemize}
\item The global Clifford algebra $\mathrm{Cl}_{X}$ on $X$: a restricted direct product over all points $P\in X$ of the local Clifford algebras $\mathrm{Cl}_{P}$, which are related to the $k$-vector spaces $\mathcal{M}_{P}$ by the residue maps $\Res_{P}(fg)$.
\item The adelic Clifford module $\mathfrak{F}_{X}$ (the global fermion Fock space): a restricted $\ZZ/2\ZZ$-graded tensor product of the local Clifford modules $\mathfrak{F}_{P}$ over all $P\in X$.
\item The `expectation value' functional: a linear map $\langle\,\cdot\,\rangle : \mathfrak{F}_{X}\rightarrow k$, satisfying the condition
\begin{equation} \label{fermi-1}
\langle f\cdot u\rangle=0\quad\text{for all}\quad f\in\mathcal{M}(L)\subset \mathrm{Cl}_{X},\;u\in\mathfrak{F}_{X},
\end{equation}
where the vector space $\mathcal{M}(L)$ is embedded diagonally into the global Clifford algebra $\mathrm{Cl}_{X}$.
\end{itemize}

In this pure algebraic formulation of one-component free fermions on an algebraic curve, the products of field operators at points $P\in X$ are replaced by the vectors $u=\hat{\otimes}_{P\in X}u_{P}\in\mathfrak{F}_{X}$, 
and the linear map $\langle\,\cdot\,\rangle$ is a mathematical way of defining the correlation functions of quantum fields. At the physical level of rigor, these functions  are introduced by the Feynman path integral. The vector space $\mathcal{M}(L)$ acts on $\mathfrak{F}_{X}$ by global symmetries, and the invariance of the quantum theory of free fermions with respect to these symmetries is expressed by the quantum conservation laws \eqref{fermi-1}, also known as the \emph{additive Ward identities}. It is proved in \cite{Witten-1, Witten-2} that if the spin structure $L$ has no global holomorphic sections, then the additive Ward identities uniquely determine the expectation value functional $\langle\,\cdot\,\rangle$. The relations \eqref{fermi-1} are compatible with the global residue theorem on $X$:
\begin{equation*}
\sum_{P\in X}\Res_{P}(fdg)=0,\quad f,g\in\mathcal{M}(L).
\end{equation*}

Witten \cite{Witten-2} developed the basics of quantum field theories associated with  current algebras on an algebraic curves and mentioned the theories associated with loop  groups on algebraic curves. The global symmetries of these theories are respectively given by the rational maps of the algebraic curve $X$ to a finite-dimensional semi-simple Lie algebra over $k$ and the rational maps of $X$ to the corresponding Lie group. In the latter case, the analogues of quantum conservation laws \eqref{fermi-1} were called \emph{multiplicative Ward identities} in \cite{Witten-2}. It was emphasized in  \cite[Sect. IV]{Witten-2} that if the genus of $X$ is greater than zero, then the  Ward identities do not uniquely determine the expectation value functional $\langle\,\cdot\,\rangle$, even in the Lie-algebraic case. Thus the main problem in the construction of quantum field theories on an algebraic curve is to find additional conditions which would uniquely determine the linear functional $\langle\,\cdot\,\rangle$.   

When $X$ is a Riemann surface (that is, an algebraic curve over the field  $\CC$ of complex numbers, equipped with the complex topology), the usual physicist's representation of correlation functions is given by the Feynman path integral, which uses the Lagrangian formulation of the theory.
This approach is not applicable in the case when $X$ is an algebraic curve over an arbitrary field of constants. Hence one needs other ways to define the correlation functions. The basic example is given by the `integrable' case, when the expectation value functional is uniquely determined by the global symmetries (and hence so are all correlation functions). 

In \cite{Takhtajan} we we gave a solution of the problem of the unique determination of the expectation value functional for the simplest scalar theories, when the finite-dimensional Lie algebra is the abelian Lie algebra $k$, and the corresponding Lie group is the multiplicative group $k^{\ast}=k\setminus\{0\}$.   
We call these quantum field theories the theories of \emph{additive} and \emph{multiplicative bosons} respectively. The solution suggested in 
\cite{Takhtajan} involves enlarging the global symmetries by considering algebraic analogue of the vector space additive multi-valued functions on a Riemann surface (analogues of the classical abelian integrals of the second kind with zero $a$-periods). Although the classical theory of abelian integrals was already developed by Riemann (see, for example, \cite{Iwasawa} and \cite{Kra} for a modern exposition), the corresponding algebraic theory (integral calculus on algebraic curves) has not been fully developed. In this paper we partially fill this gap in the case when the field $k$ of constants has characteristic zero and give an explicit construction of the quantum field theories of additive bosons on an algebraic curve. These theories are naturally connected with the algebraic de Rham theorem, and their global symmetries form a vector space of additive multi-valued functions; see Theorems \ref{additive theorem} and \ref{charged theorem} for precise statements. Our construction of quantum field theories on algebraic curves may be regarded as an algebraic analogue of the geometric realization of conformal field theories on Riemann surfaces in \cite{KNTY}. The quantum field theory of multiplicative bosons requires an algebraic analogue of the group of multiplicative multi-valued  functions on a Riemann surface (analogues of the exponentials of abelian integrals of the third kind with zero $a$-periods). We plan to discuss the analogue of this group and the corresponding multiplicative Ward identities in a separate publication.

Here is the more detailed description of the contents of the paper. In Section \ref{Basic facts} we recall the necessary basic facts from the theory of algebraic curves. Namely, let $X$ be an algebraic curve of genus $g$ over an algebraically closed field $k$ of constants, $F=k(X)$ the field of rational functions on $X$, and $F_{P}$ the corresponding local fields (the completions of $F$ with respect to the regular discrete valuations $v_{P}$ corresponding to the discrete valuation rings at points $P\in X$). In Section \ref{Definitions}
we introduce the ring of ad\`{e}les 
$$\A_{X}=\coprod_{P\in X}F_{P}$$ 
as a restricted direct product of the local fields $F_{P}$ and  describe Serre's adelic interpretation of cohomology. In Section \ref{differentials} we recall the definitions of the $F$-module $\Omega^1_{F/k}$ of K\"{a}hler differentials on $X$, the corresponding $\A_{X}$-module of differential ad\`{e}les $\bm{\Omega}_{X}$, the differential map $d:\A_{X}\rightarrow\bm{\Omega}_{X}$ and the residue map $\Res: \bm{\Omega}_{X}\rightarrow k$. In Section \ref{R-R} we describe Serre duality and the Riemann-Roch theorem. 

In Section \ref{Calculus}, assuming that the field $k$ of constants is of characteristic zero, we recall  the differential calculus on an algebraic curve $X$ (the structural theory of the $k$-vector space  of K\"{a}hler differentials $\Omega^1_{F/k}$ on $X$) and develop a corresponding integral calculus. Namely, in Section \ref{A-functions} we
follow \cite{ Chevalley} and \cite{Eichler} and endow the $k$-vector space $\Omega^{(2\text{nd})}$ of differentials of the second kind (that is, differentials on $X$ with zero residues) with a skew-symmetric bilinear form
\begin{equation*}
(\omega_1,\omega_2)_X=\sum_{P\in X}
\Res_P(d^{-1}\omega_1\,\omega_2),\quad\omega_1,\omega_2\in \Omega^{(\mathrm{2nd})}.
\end{equation*}

The main result of the differential calculus is Theorem \ref{Chevalley},  an algebraic version of de Rham theorem. Theorem \ref{Chevalley} goes back to Chevalley and Eichler and, for an algebraic curve $X$ of genus $g\geq 1$, states\footnote{The case $g=0$ is trivial.}  that the $2g$-dimensional $k$-vector space $\Omega^{(2\text{nd})}/dF$ is a symplectic vector space with symplectic form
$(~,~)_X$.
Moreover, for every choice of a non-special effective divisor $D=P_{1}+\dots +P_{g}$ of a degree $g$ on $X$ and uniformizers
$t_{i}$ at $P_{i}$,  there is an isomorphism
$$\Omega^{(\mathrm{2nd})}/dF\simeq\Omega^{(\mathrm{2nd})}\cap \Omega^1_{F/k}(2D).$$
The space $\Omega^{(\mathrm{2nd})}\cap \Omega^1_{F/k}(2D)$ has a natural symplectic basis $\{\theta_{i},\omega_{i}\}_{i=1}^{g}$, 
where $\theta_{i}$ (resp. $\omega_{i}$) are differentials of the first (reps. second) kind  with the following properties. The $\theta_{i}$ vanish at all points $P_{j}$ for $j\neq i$, and $\theta_{i}=(1+O(t_{i}))dt_{i}$ at $P_{i}$, while the $\omega_{i}$ are regular at all points $P_{j}$ for $j\neq i$, and $\omega_{i}=(t_{i}^{-2}+O(t_{i}))dt_{i}$ at $P_{i}$. Hence $\theta_i$ (reap. $\omega_i$) are algebraic analogues of the differentials
of the first kind with normalized $a$-periods (resp. differentials of the second kind with second order poles, zero $a$-periods and normalized $b$-periods) on a compact Riemann surface. The $a$-periods of $\omega\in\Omega^{(\mathrm{2nd})}$ are algebraically defined by $(\omega,\omega_{i})_{X}$, $i=1,\dots,g$, and we write $\Omega^{(\mathrm{2nd})}_{0}$ for the isotropic subspace of $\Omega^{(\mathrm{2nd})}$ consisting of all differentials of the second kind with zero $a$-periods. By Proposition \ref{Second} we have
\begin{equation} \label{2-a}
\Omega^{(\mathrm{2nd})}_0 = k\cdot\omega_1\oplus\dots\oplus k\cdot\omega_g\oplus dF.
\end{equation}

In Section \ref{A-functions} we also introduce an algebraic notion of additive multi-valued functions on $X$. By definition, the $k$-vector space of additive multi-valued functions is a subspace $\mathcal{A}(X)$ of the ad\`{e}le ring $\A_X$ satisfying $F\subset \mathcal{A}(X)$ and $d\mathcal{A}(X)\subset \Omega^1_{F/k}$ and the additional condition that if $a\in\mathcal{A}(X)$ and $da=0$, then $a=c\in k$. The main result of the integral calculus for differentials of the second kind with zero $a$-periods is the explicit construction of the vector space $\mathcal{A}(X,D)$ in Example \ref{Additive}. This space plays a fundamental role in the theory of additive bosons. It is parametrized by the choices of a non-special divisor 
$D=P_1 + \dots + P_g$  of degree $g$ on $X$,  local uniformizers $t_{i}$ at the points $P_{i}$ and solutions of the equations $d\eta_i = \omega_i$ in $\A_{X}$ (with any fixed choice of the local additive constants). It is defined as
\begin{equation*} 
\mathcal{A}(X;D) =k\cdot\eta_1\oplus\dots\oplus k\cdot\eta_g\oplus F\subset\A_X
\end{equation*}
and possesses the property $d(\mathcal{A}(X;D))=\Omega^{(\mathrm{2nd})}_0$. We finally introduce the additive multi-valued functions $\eta_{P}^{(n)}\in\mathcal{A}(X;D)$ with a single pole of order $n$ at $P\in X$ and prove (see Lemma \ref{reciprocity}) that every rational function $f\in F$ has a unique partial fraction expansion, the partial fractions being these $\eta_{P}^{(n)}$. We also explain the difficulties arising in an attempt to define algebraic analogues of multiplicative multi-valued functions.

In Section \ref{Local} we formulate the local quantum field theories of additive and charged  bosons. The local theory of additive bosons is associated with the representation theory of the local Heisenberg algebra $\mathfrak{g}_{P}$, a one-dimensional central extension of the abelian Lie algebra $F_{P}$, $P\in X$,  by the $2$-cocycle $c_{P}(f,g)=-\Res_{P}(fdg)$. In Section \ref{Heisenberg-Lie-local} we introduce the highest-weight representation $\rho$ of $\mathfrak{g}_{P}$ on the local Fock space $\cF_{P}$ and define the corresponding contragradient representation $\rho^{\vee}$ of $\mathfrak{g}_{P}$ on the dual local Fock space $\cF_{P}^{\vee}$. In Section \ref{lattice algebra-local} we define a local lattice algebra $\mathfrak{l}_{P}$ as a semi-direct sum of the local Heisenberg algebra $\mathfrak{g}_{P}$ and the abelian Lie algebra $k[\ZZ]$, the group algebra of $\ZZ$. The corresponding irreducible highest-weight $\mathfrak{l}_{P}$-module is the local Fock space $\cB_{P}$ of `charged bosons' (a tensor product of $k[\ZZ]$ and $\cF_{P}$). The material in Sections \ref{Heisenberg-Lie-local} and \ref{lattice algebra-local} is essentially standard (see \cite{kac, ben-zvi-frenkel}). 

In Section \ref{Global} we finally state the global quantum field theories, starting in Section \ref{AB} with the theory of additive bosons on an algebraic curve $X$. This theory is naturally connected with the global Heisenberg algebra $\g_X$, a one-dimensional central extension of the abelian Lie algebra $\g\mathfrak{l}_1(\A_X)=\A_X$ by the $2$-cocycle $c_X=\sum_{P\in X} c_P$. Since the subspace $\Omega_{0}^{(\mathrm{2nd})}$ is isotropic with respect to the bilinear form $(~,~)_X$, we have
\begin{equation} \label{general-res}
c_{X}(a_{1}, a_{2})=0\quad \forall\, a_{1}, a_{2}\in\mathcal{A}(X,D)\subset \A_{X}.
\end{equation}
This may be regarded as a generalized residue theorem for additive multi-valued functions.
The irreducible highest-weight module of the global Heisenberg algebra $\g_{X}$ is the global Fock space
$\cF_{X}$, a restricted tensor product of the local Fock spaces $\cF_P$ over all points $P\in X$. It may be regarded as the space of observables of the quantum theory of additive bosons on $X$.
In Theorem \ref{additive theorem} we prove that there is a unique normalized expectation value functional $\langle\, \cdot\,\rangle:\cF_X\rightarrow k$, uniquely characterized by the global symmetries  
\begin{equation} \label{AWI}
\langle\rho(a)v\rangle=0\quad \forall\, a\in \mathcal{A}(X;D),\quad v\in\cF_{X}.
\end{equation}
Here $\mathcal{A}(X;D)\subset\A_{X}$ is the vector space of additive multi-valued functions on $X$, defined in Section \ref{A-functions}, and $\rho:\g_{X}\rightarrow\End\cF_{X}$ is the corresponding representation of the global Heisenberg algebra. Specifically, we show in Theorem \ref{additive theorem} that
\begin{equation*}
\langle v\rangle=(\Omega_{X},v)\quad\text{for all}\quad v\in\cF_{X}, 
\end{equation*}
where $\Omega_{X}\in\cF_{X}^{\vee}$ is a vector in the dual space to $\cF_{X}$, satisfying an infinite system of equations 
\begin{equation}\label{system-1}
\Omega_{X}\cdot\rho^{\vee}(a)=0\quad \forall\, a\in\mathcal{A}(X,D).
\end{equation}
The vector $\Omega_{X}$ is given by an explicit formula (see Theorem \ref{additive theorem}), which encodes  the analogues of all correlation functions of  quantum additive bosons on $X$. The compatibility of the system \eqref{system-1} is based on the reciprocity law (proved in Lemma \ref{reciprocity}) for the differentials of the second kind with zero $a$-periods. 

The additive Ward identities \eqref{AWI} are also compatible with the generalized residue theorem. 
Namely, since $[\rho(x),\rho(y)]=c_{X}(x,y)\bm{I}$ for $x,y\in\A_{X}$, where $\bm{I}$ is the identity operator on $\cF_{X}$, we get from \eqref{AWI} that for $a_{1}, a_{2}\in\mathcal{A}(X,D)$,
$$0=\langle(\rho(a_{1})\rho(a_{2})-\rho(a_{2})\rho(a_{1}))v\rangle=c_{X}(a_{1},a_{2})\langle v\rangle\quad\forall\, v\in\cF_{X},$$
which yields \eqref{general-res}.

In Section \ref{CB} we define a global lattice algebra $\mathfrak{l}_X$ as a semi-direct sum of the global Heisenberg algebra $\g_{X}$ and the abelian Lie algebra $k[\Div_{0}(X)]$ with generators $e_{D}$, where $D\in\Div_{0}(X)$ is the group algebra of the additive group $\Div_{0}(X)$ of divisors of degree $0$ on $X$. Its irreducible highest weight module is 
the global Fock space $\cB_X$ of charged bosons, which is  the tensor product of the group algebra $k[\Div_0(X)]$ and
the Fock space $\cF_X$ of additive bosons. The main result of Section  \ref{CB} is Theorem \ref{charged theorem} on the existence and uniqueness of an expectation value functional $ \langle\, \cdot\,\rangle:\cB_X\rightarrow k$ which is normalized with respect to the action of the group algebra $k[\Div_0(X)]$ and satisfies the additive Ward identities \eqref{AWI} with respect to the action of the global symmetries (additive multi-valued functions in $\mathcal{A}(X,D)$) on the global Fock space $\cB_{X}$. This functional is of the form $\langle v\rangle=(\hat{\Omega}_{X},v)$ where the vector $\hat{\Omega}_{X}\in\cB_{X}^{\vee}$ in the space dual to $\cB_{X}$ is given by an explicit formula (see Theorem \ref{charged theorem}), which encodes  all correlation functions of quantum charged additive bosons on $X$. In Section \ref{Inv}, following a suggestion of the referee, in Section \ref{Inv} we give a more invariant formulation of Theorem \ref{additive theorem}.
\vspace{3mm}

\noindent
I am grateful to the referee for his careful reading of the manuscript, constructive criticism, remarks,  and valuable suggestions. 

\section{Basic Facts\label{Basic facts}}
Here we recall necessary facts from the theory of algebraic curves. This material is essentially standard (see \cite{Chevalley,Serre,Iwasawa}).
\subsection{Definitions\label{Definitions}} 
An algebraic curve $X$ over an algebraically closed field $k$ is an
irreducible non-singular one-dimensional projective variety over $k$. It is equipped with
the Zariski topology. The field  $F=k(X)$ of rational functions on $X$ is a
finitely generated extension of $k$ of transcendence degree $1$. Conversely, every finitely generated extension of $k$ of transcendence degree $1$ corresponds to a unique (up to isomorphism) algebraic curve over
$k$. Closed points $P$ on $X$ correspond to discrete valuation rings $O_P$ (subrings
of $F$). The rings $O_P$ for all $P\in X$ form a sheaf of rings on $X$: the structure sheaf $\mathcal{O}_X$, a subsheaf of the constant sheaf $\underline{F}$.

For every point $P\in X$ let $v_P$ be the regular discrete valuation of $F$
over $k$, corresponding to the discrete valuation ring $O_P$. The completion of $F$ with respect to $v_P$ is a complete closed field $F_P$ with valuation ring
$\mathcal{O}_P$, which is the completed local ring at $P$ with prime ideal $\mathfrak{p}$ 
and residue class field $k =\mathcal{O}_P/\mathfrak{p}$. The ring $\A_X$ of ad\`{e}les
of $X$ is
\begin{equation*}
\A_X=\coprod_{P\in X}F_P,
\end{equation*}
the restricted direct product over all points $P\in X$ of the local fields $F_P$
with respect to the local rings $\mathcal{O}_P$. By definition,
\begin{displaymath}
x=\{x_P\}_{P\in X}\in \A_X
\end{displaymath}
if $x_P\in\mathcal{O}_P$  for all but finitely many $P\in X$.
The field $F$ is contained in all local fields $F_P$ and is diagonally embedded in
$\A_X$:
\begin{displaymath}
F\ni f\mapsto\{f|_{P}\}_{P\in X}\in\A_X.
\end{displaymath}

The divisor group $\Div(X)$ of $X$ is the free abelian group generated by the points $P\in
X$. By definition,
\begin{displaymath}
D=\sum_{P\in X}n_P\cdot P\in\Div(X)
\end{displaymath}
if $n_P=v_P(D)\in\ZZ$ and $n_P=0$ for all but finitely many $P\in X$. 
Divisors of the form
\begin{equation*}
(f)=\sum_{P\in X} v_P(f)\cdot P\in\Div(X),
\end{equation*}
where $f\in F^{\ast}=F\setminus\{0\}$, are called \emph{principal divisors}. They form a subgroup $\mathrm{PDiv}(X)\simeq F^{\ast}/k^{\ast}$ of $\Div(X)$. The degree of a divisor $D$ is
\begin{equation*}
\deg D=\sum_{P\in X}n_P=\sum_{P\in X} v_P(D)\in\ZZ,
\end{equation*}
and we have $\deg(f)=0$ for $f\in F^{\ast}$. A divisor $D$ is said to be \emph{effective}, if $v_P(D)\geq 0$ for all $P\in X$.
By definition, $D_1$ and $D_2$ are linear equivalent ($D_1\sim D_2$) if $D_1 - D_2 = (f)$ for some $f\in F^{\ast}$. The equivalence classes of divisors form the divisor class group $\mathrm{Cl}(X)=\Div(X)/\mathrm{PDiv}(X)$.

For every divisor $D$ we define a subspace $\A_X(D)$ of the $k$-vector space $\A_X$ by putting
\begin{equation*}
\A_X(D)=\{x\in\A_X : v_P(x_P)\geq -v_{P}(D)~\forall\, P\in X\}\,.
\end{equation*}
The ring  $\A_X$ of ad\`{e}les is a topological ring with the product topology. A base of
neighborhoods of 0 is given by the subspaces $\A_X(D),\,D\in\Div(X)$, and $\A_X$ is a
$k$-vector space with linear topology in the sense of Lefschetz \cite[Ch.~II, \S
6]{Lefschetz}. Since subspaces $\A_X(D)$ is linear compact, $\A_X$ is locally
linear compact. The $k$-vector space $F=k(X)$ is discrete in $\A_X$ and the quotient
space $\A_X/F$ is linear compact \cite[App., \S3] {Iwasawa}.

For every divisor $D$ we have an algebraic coherent sheaf $\mathcal{O}_{X}(D)$ on $X$ whose
stalk at any point $P\in X$ is
\begin{equation*}
\mathcal{O}_{X}(D)_P=\{f\in F : v_P(f)\geq - v_P(D)\}.
\end{equation*}
Linear equivalent divisors correspond to isomorphic sheaves. We denote the \v{C}ech cohomology groups of the sheaf by $\mathcal{O}_{X}(D)$ $H^i(X,\mathcal{O}_{X}(D))$ (these are finite-dimensional vector spaces over $k$, trivial for $i>1$) and put $h^i(D)=\dim_k H^i(X,\mathcal{O}_{X}(D))$. The zero divisor $D=0$ corresponds to the
structure sheaf $\mathcal{O}_X$. In this case, $h^0(0)=1$ and $h^1(0)=g$ is the
arithmetic genus of $X$. We have
\begin{equation*}
H^0(X,\mathcal{O}_{X}(D))= \A_X(D)\cap F, \quad H^1(X,\mathcal{O}_{X}(D))\simeq\A_X/(\A_X(D)+ F),
\end{equation*}
which is Serre's adelic interpretation of cohomology \cite[Ch. II, \S5]{Serre}.

\subsection{Differentials and residues\label{differentials}} 
Let $R$ be a ring over $k$. The module $\Omega^1_{R/k}$ of K\"{a}hler differentials of $R$ is
the universal $R$-module with the property that there is a $k$-linear map $d:R\rightarrow\Omega^1_{R/k}$
satisfying the Leibniz rule 
$$d(fg)=fdg + gdf, \quad f,g\in R.$$

When $R=F$ is the field $k(X)$ of rational functions on an algebraic curve $X$, $\Omega^1_{F/k}$ is a one-dimensional vector space over $F$. 
Let $t\in F$ be a local coordinate at $P$ in the Zariski topology,  that is, a rational function on $X$ with $v_P(t)=1$.
Then $dt$ is a basis of the $F$-vector space $\Omega^1_{F/k}$, that is, every
K\"{a}hler differential can be written as $\omega=fdt$ for some $f\in F$. The order of
$\omega\in\Omega^1_{F/k}$ at $P$ is defined by
\begin{equation*}
v_P(\omega)=v_P(f).
\end{equation*}
It is independent of the choice of the local coordinate at $P$ and determines a
valuation on $\Omega^1_{F/k}$.

The family of $O_P$-modules $\Omega^1_{O_P/k}$ for all points $P\in X$ forms an algebraic
coherent sheaf $\underline{\Omega}$, a subsheaf of the constant sheaf
$\underline{\Omega^1_{F/k}}$. Moreover,
\begin{equation*}
\Omega^1_{F/k} = \Omega^1_{O_P/k}\underset{{O}_P}{\otimes} F.
\end{equation*}

When $k$ has characteristic 0,  the
$F_P$-module $\Omega^1_{F_P/k}$ is an infinite-dimensional $F_P$-vector space for every point $P\in X$ (the
map $d$ is not continuous with respect to the $\mathfrak{p}$-adic topology on
$F_P$). Following \cite[Ch. II, \S 11]{Serre}, we define
\begin{equation*}
\tilde\Omega^1_{F_P/k} = \Omega^1_{F_P/k}/\mathcal{Q},
\end{equation*}
where $\mathcal{Q}=\cap_{n\geq 0}\,\pp^n d(\mathcal{O}_P)$ and, therefore, 
$\dim_{F_P}\tilde\Omega^1_{F_P/k}=1$.
The resulting $F_P$-module $\tilde\Omega^1_{F_P/k}$ is the completion of the $F$-module
$\Omega^1_{F/k}$ with respect to the valuation $v_P$. The completion of the
$O_P$-module $\Omega^1_{O_P/k}$ is the $\mathcal{O}_P$-module
$\tilde\Omega^1_{\mathcal{O}_P/k}$ and
\begin{equation*}
\tilde\Omega^1_{F_P/k} =
\tilde\Omega^1_{\mathcal{O}_P/k}\underset{\mathcal{O}_P}{\otimes} F_P.
\end{equation*}

We define the $\A_X$-module  $\bm{\Omega}_X$ of differential ad\`{e}les of the sheaf $\underline{\Omega}$ by the formula
\begin{equation*}
\bm{\Omega}_X=\coprod_{P\in X}\tilde\Omega^1_{F_P/k}.
\end{equation*}
This is the restricted direct product over all points $P\in X$ of the $F_P$-modules
$\tilde\Omega^1_{F_P/k}$ with respect to the $\mathcal{O}_P$-modules
$\tilde\Omega^1_{\mathcal{O}_P/k}$. The $F$-module $\Omega^1_{F/k}$ is contained in
all $F_P$-modules $\tilde\Omega^1_{F_P/k}$ and is diagonally embedded into
$\bm{\Omega}_X$:
\begin{equation*}
\Omega^1_{F/k}\ni \omega\mapsto\{\omega|_{P}\}_{P\in X} \in\bm{\Omega}_X.
\end{equation*}
The $k$-vector space $\bm{\Omega}_X$ has a linear topology with a base of neighborhoods of zero given by the subspaces $\bm{\Omega}_X(D)$ for all $D\in\Div(X)$:
\begin{equation*}
\bm{\Omega}_X(D)= \{\omega=\{\omega_P\}_{P\in X}
\in\boldsymbol{\Omega}_X : v_P(\omega_P)\geq -v_P(D)\;\;\forall\, P\in
X\}.
\end{equation*}
This topological space is locally linear compact. The maps $d: F_P\rightarrow\tilde\Omega^1_{F_P/k}$ for all $P\in X$ determine a continuous map
$d:\A_X\rightarrow \bm{\Omega}_X$
satisfying the Leibniz rule.
\begin{remark}
The $\A_X$-module $\bm{\Omega}_X$ is essentially the set of `principal part
systems of degree $1$' on $X$ in the sense of Eichler (see \cite[Ch.~III, \S
5.2]{Eichler}).
\end{remark}
Take $\omega \in\tilde\Omega^1_{F_P/k}$, and let $t$ be a local parameter of the field
$F_P$, so that $dt$ is a basis of the $\mathcal{O}_P$-module $\tilde\Omega^1_{\mathcal{O}_P/k}$. The
residue map $\Res_P: \tilde\Omega^1_{F_P/k}\rightarrow k$ is defined as
\begin{equation*}
\Res_P(\omega)=c_{-1},\quad\text{where}\quad\omega=\sum_{n\gg -\infty}^{\infty}c_n t^n dt,
\end{equation*}
and the symbol $n\gg-\infty$ means that the summation is taken over only finitely many negative
values of $n$. The definition of the residue is independent of the choice of the local
parameter. The residue map is continuous with respect to the $\mathfrak{p}$-adic topology on
$\tilde\Omega^1_{F_P/k}$ and the discrete topology on $k$. The local residue maps
$\Res_P$ give rise to the global residue map $\Res: \bm{\Omega}_X\rightarrow
k$,
\begin{equation*}
\Res\,\omega=\sum_{P\in X}\Res_P(\omega_P),\;\;\omega=\{\omega_P\}_{P\in X}
\in\bm{\Omega}_X.
\end{equation*}
The global residue map is well-defined, continuous, and possesses the following
fundamental property.
\begin{theorem}[The residue formula] For every $\omega\in\Omega^1_{F/k}$,
\begin{displaymath}
\Res\,\omega=\sum_{P\in X}\Res_P(\omega|_P) = 0.
\end{displaymath}
\end{theorem}
\subsection{Serre's duality and the Riemann-Roch theorem\label{R-R}}
We put
\begin{equation*}
\Omega^1_{F/k}(D)= \Omega^1_{F/k}\cap\,\bm{\Omega}_X(D)=
\{\omega\in\Omega^1_{F/k} : v_P(\omega)\geq -v_P(D)~\forall\, P\in X\}
\end{equation*}
and define the residue pairing $(~,~):\bm{\Omega}_X\otimes_K\A_X\rightarrow k$ by the formula
\begin{equation*}
(\omega,x)=\sum_{P\in
X}\Res_P(x_P\omega_P),~\text{where}~\omega\in\bm{\Omega}_X,\, x\in\A_X.
\end{equation*}
The residue pairing has the following properties:
\begin{enumerate}
\item[P1)] $(\omega,x)=0$ if $\omega\in\Omega^1_{F/k}$ and $x\in F$,
\item[P2)] $(\omega,x)=$ if $\omega\in\bm{\Omega}_X(-D)$ and $x\in\A_X(D)$.
\end{enumerate}
It follows from P1), P2) that for
every $D\in\Div(X)$ the formula $\imath(\omega)(x)=(\omega,x)$  determines a $k$-linear map
\begin{equation*}
\imath: \Omega^1_{F/k}(-D)\rightarrow \left(\A_X/(\A_X(D) + F)\right)^\vee,
\end{equation*}
where $V^\vee=\Hom(V,k)$ is the topological dual of a $k$-vector space $V$ with linear
topology.
\begin{theorem}[Serre's duality] For every $D\in\Div(X)$ the map $\imath$ is
an isomorphism. Hence the finite-dimensional $k$-vector spaces $\Omega^1_{F/k}(-D)$ and
$\A_X/(\A_X(D) + F)$ are dual with respect to the residue pairing.
\end{theorem}
\begin{corollary}[The strong residue theorem] \mbox{}
\begin{enumerate}
\item[(i)] An ad\`{e}le $x\in\A_X$ corresponds to a rational function on $X$ under
the embedding $F\hookrightarrow\A_X$ if and only if $(\omega,x)=0$ for all $\omega\in\Omega^1_{F/k}$.
\item[(ii)] A differential ad\`{e}le $\omega\in\bm{\Omega}_X$ corresponds to a K\"{a}hler differential
on $X$ under the embedding $\Omega^1_{F/k}\hookrightarrow\bm{\Omega}_X$ if and
only if $(\omega,f)=0$ for all $f\in F$.
\end{enumerate}
\end{corollary}
\begin{proof}
Suppose that the condition (i) holds. It follows from Serre's duality that $x\in\A_X(D) + F$ for every
$D\in\Div(X)$ and, since $F\cap\A_X(D)=0$ for $D<0$, we have $x\in F$. To prove (ii), take $\omega_0\in\Omega^1_{F/k},\,\omega_0\neq 0$. Putting $x=\omega/\omega_0\in\A_X$,
we have $0=(\omega,f)=(f\omega_0,x)$ for all $f\in F$, whence $x\in F$ by part (i).
\end{proof}
\begin{remark}
The strong residue theorem is stated in a slightly different form in \cite[Ch. III, \S 5.3]{Eichler}.
\end{remark}
For $\omega\in\Omega^1_{F/k}$ we put
\begin{equation*}
(\omega)=\sum_{P\in X}v_P(\omega)\cdot P\in\Div(X).
\end{equation*}
Since $\dim_F \Omega^1_{F/k}=1$, all divisors $(\omega)$ are linear equivalent and determine a divisor class $K\in\text{Cl}(X)$, the canonical class of $X$. The following result is obtained by combining the Riemann-Roch formula for the Euler characteristic of a divisor $D$:
\begin{equation*}
\chi(D)=h^0(D) - h^1(D)=\deg D + 1 - g,
\end{equation*}
with Serre duality and the adelic interpretation of cohomology.
\begin{theorem}[Riemann-Roch theorem] For every $D\in\Div(X)$ we have
\begin{equation*}
h^0(D) - h^0(K-D) = \deg D + 1-g.
\end{equation*}
\end{theorem}
An effective divisor $D$ on $X$ is called \emph{non-special} if $h^0(K-D)=0$. It follows from
the Riemann-Roch theorem that an effective divisor $D$ of degree $g$ is non-special if
and only if $h^0(D)=1$. In other words, the only rational functions whose poles are contained in an effective non-special divisor of degree $g$ are constant functions. 

\section{Differential and Integral Calculus\label{Calculus}}

From now on we assume that the algebraically closed field $k$ has characteristic 0 and the algebraic curve $X$ has genus $g\geq1$.
\subsection{Differentials of the second kind and `additive functions'\label{A-functions}}
\label{second kind}
Following the classical terminology, we call a K\"{a}hler differential $\omega\in\Omega^1_{F/k}$
\emph{a differential of the second kind} if $\Res_P \omega = 0$ for all $P\in X$. The
$k$-vector space $\Omega^{(2\text{nd})}$ of differentials of the second kind on $X$ carries
a canonical skew-symmetric bilinear form $(~,~)_X$ defined as follows. For every
$\omega\in\Omega^{(\mathrm{2nd})}$ let $x=\{x_P\}_{P\in X}\in\A_X$ be an ad\`{e}le satisfying the equality
\begin{equation*}
d\,x_P = \left.\omega\right|_P\;\;\forall\, P\in X.
\end{equation*}
For every $P\in X$ there is a unique (up to an additive constant
in $k$) element $x_P\in F_P$ with this property, and we have $x_P\in\mathcal{O}_P$ for all but finitely many $P\in X$. We define
$x=d^{-1}\omega$ and put
\begin{equation*}
(\omega_1,\omega_2)_X=\sum_{P\in X}
\Res_P(d^{-1}\omega_1\,\omega_2),\;\;\omega_1,\omega_2\in \Omega^{(\mathrm{2nd})}.
\end{equation*}
The bilinear form $(~,~)_X$ is independent of the choice of the additive constants in the
definition of $d^{-1}$ and is skew-symmetric. When $X$ is a Riemann surface, the bilinear form $(~,~)_X$ corresponds to the standard pairing in the cohomology under the isomorphism $\Omega^{(\mathrm{2nd})}/dF\simeq H_{\mathrm{dR}}^{1}(X)$ (see \cite[Ch. III, \S5]{Griffiths-Harris}).

The infinite-dimensional $k$-vector space $\Omega^{(\mathrm{2nd})}$ has a $g$-dimensional
subspace $\Omega^{(\mathrm{1st})}=\Omega^1_{F/k}(0)$ of the differentials of the first kind.
The infinite-dimensional subspace $\Omega^{(\mathrm{1st})}\oplus dF$ of $\Omega^{(\mathrm{2nd})}$ is isotropic with respect
to the bilinear form $(~,~)_X$. Since there is no canonical choice of the isotropic complementary subspace to
$\Omega^{(\mathrm{1st})}\oplus dF$ in $\Omega^{(\mathrm{2nd})}$, the exact sequence
\begin{equation*}
0\rightarrow \Omega^{(\mathrm{1st})}\oplus dF\rightarrow\Omega^{(\mathrm{2nd})}
\rightarrow\Omega^{(\mathrm{2nd})}/(\Omega^{(\mathrm{1st})}\oplus dF)\rightarrow 0
\end{equation*}
does not split canonically. Nevertheless we have the following
fundamental result (see \cite[Ch. VI, \S 8]{Chevalley} and \cite[Ch. III, \S\S 5.3-5.4]{Eichler}), which may be regarded as an algebraic de Rham theorem.  
\begin{theorem} \label{Chevalley} \mbox{}
\begin{enumerate}
\item[(i)] The restriction of the bilinear form $(~,~)_X$ to $\Omega^{(\mathrm{2nd})}/dF$ is non-degenerate and
\begin{displaymath}
\dim_k \Omega^{(\mathrm{2nd})}/dF =2g.
\end{displaymath}
\item[(ii)]
For every effective non-special effective divisor $D$ on $X$ of degree $g$ there is an isomorphism
\begin{equation*}
\Omega^{(\mathrm{2nd})}/dF\simeq\Omega^{(\mathrm{2nd})}\cap \Omega^1_{F/k}(2D).
\end{equation*}
\item[(iii)] Let $D=P_1 + \dots +P_g$ be a non-special divisor with distinct points. Then every choice of local uniformizers $t_i$ at $P_i$ determines a symplectic basis $\{\theta_i,\omega_i\}_{i=1}^g$ of the $k$-vector space $\Omega^{(\mathrm{2nd})}\cap \Omega^1_{F/k}(2D)$ with respect to the symplectic from $(~,~)_X$:
\begin{equation*}
(\theta_i,\theta_j)_X=(\omega_i,\omega_j)_X=0,\;\;(\theta_i,\omega_j)_X=
\delta_{ij},\quad i,j=1,\dots,g.
\end{equation*}
This basis consists of differentials $\theta_i$ of the first kind and differentials $\omega_i$ of the second kind which are uniquely determined by the conditions
\begin{equation*}
v_{P_i}\left(\theta_j -\delta_{ij}dt_i\right)>0\;\;\text{and}\;\;
v_{P_i}\left(\omega_j -\delta_{ij}t^{-2}_i dt_i\right)>0,
\end{equation*}
where $i,j=1,\dots,g$.
\item[(iv)] The subspace $k \cdot\omega_1\oplus\dots\oplus k\cdot\omega_g$ is an
 isotropic complement to $\Omega^{(\mathrm{1st})}\oplus dF$ in $\Omega^{(\mathrm{2nd})}$.
\end{enumerate}
\end{theorem}
\begin{proof} Let $(\omega)_{\infty}=n_{1}Q_{1}+\dots+n_{l}Q_{l}$ be the polar divisor of $\omega\in\Omega^{(\mathrm{2nd})}$. Since $\mathrm{char}\,k=0$, for every $Q_{i}$ there is an $f_{i}\in F$ such that $v_{Q_{i}}(\omega-df_{i})\geq 0$. We define $x=\{x_{P}\}_{P\in X}\in\A_{X}$ by the formulae
\begin{equation*}
x_{P}=\begin{cases} \left. f_{i}\right|_{Q_{i}}, & P=Q_{i},\quad i=1,\dots,l,\\
0, & P\neq Q_{i},\quad i=1,\dots,l.
\end{cases}
\end{equation*}
Since $D$ is a non-special divisor of degree $g$, we have $\Omega^{1}_{F/k}(D)=\{0\}$ and, by Serre duality, $\A_{X}(D)+F=\A_{X}$. Thus there in an $f\in F$ with the property $v_{P}(f-x)\geq -v_{P}(D)$ for all $P\in X$, whence $(\omega-df)\geq -2D$. Since $D$ is non-special such an $f$ is unique. This proves part (ii). 

To prove (i), we observe that by the Riemann-Roch theorem, 
$$\dim_{k}\Omega^{1}_{F/k}(2D)=3g-1,\quad \dim_{k}\Omega^{1}_{F/k}(D)=2g-1.$$ 
Let $\Omega^{(\mathrm{3rd})}$ be the $k$-vector space of differentials of the third kind. This subspace of $\Omega^{1}_{F/k}$ is formed by the differentials with only simple poles. Since $\Omega^{(\mathrm{2nd})}\cap\Omega^{(\mathrm{3rd})}=\Omega^{(\mathrm{1st})}$ and $\Omega^{(\mathrm{3rd})}\cap \Omega^1_{F/k}(2D)=\Omega^1_{F/k}(D)$, we conclude that
\begin{gather*} \dim_{k}\Omega^{(\mathrm{2nd})}\cap\Omega^{1}_{F/k}(2D)+\dim_{k}\Omega^1_{F/k}(D) \\=\dim_{k} \Omega^1_{F/k}(2D)+\dim_{k}\Omega^{(\mathrm{1st})}.
\end{gather*}
Using (ii), we have
$$\dim_{k}\Omega^{(\mathrm{2nd})}/dF=(3g-1)-(2g-1)+g=2g.$$

To complete the proof, we define a $k$-linear map
$$L: \Omega^{(\mathrm{2nd})}\cap \Omega^1_{F/k}(2D)\rightarrow k^{2g}$$ 
by the formula $L(\omega)=(\alpha_{1}(\omega),\dots,\alpha_{g}(\omega),\beta_{1}(\omega),\dots,\beta_{g}(\omega))$, where 
\begin{equation*}
v_{P_i}\bigl(\omega -(\alpha_i(\omega)t^{-2}_i + \beta_i(\omega)dt_i)\bigr) > 0,\quad i=1,\dots,g.
\end{equation*}
Since $D$ is non-special, $L$ is an injective map and hence an isomorphism.
The differentials $\omega_i$ and $\theta_i$ are obtained by choosing the only non-zero component of $L$ to be $\alpha_i=1$ and $\beta_i=1$ respectively.
\end{proof}
\begin{remark} The choice of a non-special effective divisor $D=P_{1}+\dots +P_{g}$ on $X$ with distinct points $P_{i}$ and uniformizers $t_{i}$ may be regarded as
an algebraic analogue of the choice of  $a$-cycles on a compact Riemann surface
of genus $g\geq 1$. Correspondingly, the differentials $\theta_i$ are analogues of differentials
of the first kind with normalized $a$-periods, and differentials $\omega_i$ are
analogues of differentials of the second kind with second order poles, zero
$a$-periods and normalized $b$-periods. The symplectic property of the basis
$\{\theta_i,\omega_i\}_{i=1}^g$ is an analogue of the reciprocity law for
differentials of the first kind and the second kind (see \cite[Ch.~5, \S1]{Iwasawa} and \cite[Ch.~VI, \S3]{Kra}).
\end{remark}
\begin{remark} It is not necessary to require all the points of the non-special effective divisor $D$ of degree $g$ to be distinct. Theorem \ref{Chevalley} and all other results in this paper can be easily modified to include divisors with multiple points.
\end{remark}
A differential $\omega$ of the second kind is said to have zero $a$-periods if
\begin{displaymath}
(\omega,\omega_i)_X=0,\quad i=1,\dots,g.
\end{displaymath}
It follows from Theorem \ref{Chevalley} that differential of the first kind with zero
$a$-periods is zero. The vector space $\Omega^{(\mathrm{2nd})}_0$ of differentials of the
second kind with zero $a$-periods has the following properties.
\begin{proposition} \label{Second} \mbox{}
\begin{itemize}
\item[(i)] The $k$-vector space $\Omega^{(\mathrm{2nd})}_0$ is an isotropic complement of $\Omega^{(\mathrm{1st})}$ in $\Omega^{(\mathrm{2nd})}$ and
\begin{equation*}
\Omega^{(\mathrm{2nd})}_0 = k\cdot\omega_1\oplus\dots\oplus k\cdot\omega_g\oplus dF.
\end{equation*}
\item[(ii)] For every $P\in X$ the $k$-vector space $\Omega_0(\ast\,P)$ of differentials of the second kind with zero $a$-periods and the only pole at $P$ has a natural filtration
\begin{displaymath}
\{0\} = \Omega_0(P)\subset\Omega_0(2P)\dots\subset \Omega_0(n P)\subset\dots,
\end{displaymath}
\begin{displaymath}
\dim_k \Omega_0(n P)=n-1.
\end{displaymath}
\item[(iii)] There is a direct sum decomposition
\begin{displaymath}
\Omega^{(\mathrm{2nd})}_0 = \bigoplus_{P\in X} \Omega_0 (\ast\,P).
\end{displaymath}
\item[(iv)] Every differential $\omega\in\Omega_0(n P)$ can be written uniquely as 
\begin{equation*}
\omega = d f + \sum_{i=1}^g c_i\omega_i,
\end{equation*}
where $f\in H^{0}(X,\mathcal{O}_{X}(D+(n-1)P))$.
\end{itemize}
\end{proposition}
\begin{proof} Part (i) follows from Theorem \ref{Chevalley} because $D$ is
non-special. Since $\dim_k \Omega^1_{F/k}(n
P)=n-1+g$, part (ii) follows from the decomposition
\begin{displaymath}
\Omega^1_{F/k}(n P) = \Omega_0(nP)\oplus\Omega^{(\mathrm{1st})}.
\end{displaymath}
Part (iii) follows from part (ii) because every differential $\omega\in\Omega^{(\mathrm{2nd})}_0$ can
be uniquely written as the sum of its principal parts at the poles. Since the divisor $D=P_1 +\dots + P_g$ is
non-special, we have 
$h^0(D+(n-1)P)= n$, and part (iv)
also follows from Theorem \ref{Chevalley}.
\end{proof}
\begin{definition} A space of \emph{additive multi-valued functions on} $X$ (additive functions for brevity) is a subspace $\mathcal{A}(X)\subset\A_X$ with the following properties.
\begin{itemize}
\item[AF1)] $F\subseteq\mathcal{A}(X)$.
\item[AF2)] If $a \in\mathcal{A}(X)$, then $da=\omega\in\Omega^1_{F/k}$
(and hence $\omega\in\Omega^{(\mathrm{2nd})}$).
\item[AF3)] If $a\in\mathcal{A}(X)$ and $da=0$, then $a=c\in k$. 
\end{itemize}
\end{definition}
\begin{remark} For every differential $\omega\in\Omega^{(\mathrm{2nd})}$, the corresponding ad\`{e}le $a=\{a_P\}_{P\in X}=d^{-1}\omega$ is
determined uniquely up to the choice of additive constants for every $P\in X$. Condition
AF3) guarantees that for all $f\in F$ these constants are
compatible with the equation $f = d^{-1}(d f) +c$. 
\end{remark}
\begin{example} \label{Additive} 
Given any non-special effective divisor $D=P_1 + \dots + P_g$  of degree $g$
on $X$ with distinct points $P_{i}$, and any choice of the local uniformizers $t_{i}$ at $P_{i}$, we have the following space $\mathcal{A}(X;D)$ of additive functions with zero $a$-periods. 
Let $\eta_i\in\A_X$ be solutions of the equations 
$$d\eta_i = \omega_i, \quad i=1,\dots,g, $$ 
with any fixed choice of the additive constants at all points $P\in X$. 
Since the divisor $D$ is non-special, the subspaces
$k\cdot\eta_1\oplus\dots\oplus k\cdot\eta_g$ and $F$ of the $k$-vector space $\A_{X}$ have zero intersection. Their direct sum
\begin{equation} \label{additive space}
\mathcal{A}(X;D) =k\cdot\eta_1\oplus\dots\oplus k\cdot\eta_g\oplus F \subset\A_X
\end{equation}
possesses properties AF1)-AF3) and the map $d: \mathcal{A}(X;D) \rightarrow\Omega^{(\mathrm{2nd})}_0$ is surjective. Indeed,
by Proposition \ref{Second} every differential 
$\omega\in\Omega^{(\mathrm{2nd})}_0$ can be written uniquely in the form
\begin{equation} \label{dec-1}
\omega = df + \sum_{i=1}^g c_i\omega_i,
\end{equation}
whence 
\begin{equation} \label{dec-2}
a=d^{-1}\omega = f + \sum_{i=1}^g c_i\eta_i + c  \in\mathcal{A}(X;D).
\end{equation}
\end{example}
\begin{remark}
The additive functions $a=d^{-1}\omega\in\mathcal{A}(X,D)$ are algebraic analogues of abelian integrals of the second kind with zero $a$-periods on a compact Riemann surface of genus $g$ (see, e.g., \cite[Ch.~V, \S2]{Iwasawa}). We can define
$$\int_{P}^{Q}\omega =a(Q)-a(P),$$
where $a(P)=a_{P}\!\!\!\mod\mathfrak{p}\in k$ for every $P\in X$.
\end{remark}
It is quite remarkable that using the additive functions in Example \ref{Additive}, one can naturally define the uniformizers $t_{P}$ at all points $P\in X$. They are uniquely determined by the following data: a choice of a non-special divisor  $D=P_1 + \dots + P_g$ with distinct points, uniformizers $t_{i}$ at $P_{i}$ and additive functions $\eta_{1},\dots,\eta_{g}$. For every $P\in X$ let $\omega^{(2)}_{P}\in \Omega_{0}(2P)$ be the unique differential of the second kind with the only second order pole at $P$ and zero
$a$-periods such that
\begin{equation} \label{norm}
\sum_{i=1}^{g}\big(\theta_{i},\omega^{(2)}_{P}\big)_{X}=1.
\end{equation}
In particular, $\omega^{(2)}_{P_i}=\omega_i$ for $i=1,\dots, g$.
Let $\eta_P=d^{-1}\omega^{(2)}_P\in\mathcal{A}(X;D)$ be an additive
function with the only simple pole at $P\in X$. 
By \eqref{dec-2}, $\eta_P$ is uniquely determined up to an overall additive constant. We fix this constant by requiring that the sum of constant terms of 
$\left.\eta_P\right|_{P_{i}}\in k((t_{i}))$ over all $i=1,\dots,g$ be equal to zero. In particular, $\eta_{P_i} =\eta_i+c_{i}$ for some $c_{i}\in k$. 
For every $P\in X$ we now define the uniformizer $t_{P}$  by the formula
$$t_{P}=
-\left.\frac{1}{\eta_P}\right|_{P}, 
$$ 
and for $\omega^{(2)}_{P}=d\eta_{P}$ we have
$$\left.\omega^{(2)}_{P}\right|_{P}=t_{P}^{-2}dt_{P},\quad P\in X.$$
Extending this construction, we now endow the subspace $\Omega_0(\ast\,P)$ for every $P\in X$ with a basis $\{\omega^{(n+1)}_P\}_{n=1}^\infty$ consisting of
differentials of the second kind with the only pole at $P$ of order $n+1$ and zero $a$-periods, where the differentials $\omega^{(2)}_{P}$ are already specified by \eqref{norm}. Let $\eta^{(n)}_P=d^{-1}\omega^{(n+1)}_P\in\mathcal{A}(X;D)$ be an additive
function with the only pole at $P\in X$ of order $n$ and with the following choice of the overall
additive constant in \eqref{dec-2}. We put
$\eta^{(1)}_P=\eta_{P}$ and require the constant term of $\left.\eta^{(n)}_P\right|_{P}\in k((t_{P}))$ to be equal to  zero for all $\eta^{(n)}_P$ with $n>1$. For every $P\in X$ let $\mathcal{A}_P(X,D)$ be the $k$-span of $\eta^{(n)}_P$, $n\in\NN$. We have a decomposition
\begin{equation} \label{AD}
\mathcal{A}(X,D)=\left(\bigoplus_{P\in X}\mathcal{A}_P(X,D)\right)\oplus k.
\end{equation}
One can restate he property of isotropy of the subspace
$\Omega_0^{(\mathrm{2nd})}=d\,\mathcal{A}(X;D)$ and the condition AF3) in the following way.
\begin{lemma} \label{reciprocity}\mbox{}
\begin{itemize}
\item[(i)] For all $P,Q\in X$ and $m,n\in\NN$ we have
\begin{equation*}
\Res_P(\eta_P^{(m)}d\,\eta_Q^{(n)}) =
\Res_Q(\eta_Q^{(n)}d\,\eta_P^{(m)}).
\end{equation*}
\item[(ii)]
Every rational function $f\in F$ admits a unique `partial fraction expansion'
\begin{equation*}
f = \sum_{i=1}^l\sum_{j=1}^{n_i} c_{ij}\eta^{(j)}_{Q_i} + c,
\end{equation*}
where $n_1 Q_1 +\dots n_l Q_l=(f)_\infty$ is the polar divisor of $f$ and
$c,c_{ij}\in k$.
\end{itemize}
\end{lemma}
\begin{proof} Since $\Res_{Q}(da)=0$ for all $a\in\A_{X}$, we get, for $P\neq Q$,
\begin{align*}
0=(\omega_{P}^{(m+1)},\omega_{Q}^{(n+1)})_{X} &=\Res_{P}(\eta_{P}^{(m)}d\eta_{Q}^{(n)}) + \Res_{Q}(\eta_{P}^{(m)}d\eta_{Q}^{(n)}) \\
& =\Res_{P}(\eta_{P}^{(m)}d\eta_{Q}^{(n)}) - \Res_{Q}(\eta_{Q}^{(n)}d\eta_{P}^{(m)}).
\end{align*}
For $P=Q$ we have $0=(\omega_{P}^{(m+1)},\omega_{P}^{(n+1)})_{X} =\Res_{P}(\eta_{P}^{(m)}d\eta_{P}^{(n)})$ for all $m,n\in\NN$.
Part (ii) follows directly from AF3) since there are $c_{ij}\in k$ such that
\begin{equation*}
df- \sum_{i=1}^l\sum_{j=1}^{n_i} c_{ij}\omega^{(j+1)}_{Q_i}\in\Omega_0^{(\mathrm{2nd})}\cap\Omega^{(\mathrm{1st})}=\{0\}.
\qedhere
\end{equation*}
\end{proof}
\begin{remark} Part (i) of Lemma \ref{reciprocity} is an algebraic analogue of the classical reciprocity law for differentials of the second kind with zero $a$-periods on a compact Riemann surface (see, e.g., \cite[Ch.~V, \S1]{Iwasawa} and \cite[Ch.~VI, \S3]{Kra}).
\end{remark}
\begin{remark}\label{0-a} In the genus zero case $X=\mathbb{P}^{1}_{k}=k\cup\{\infty\}$ we have $F=k(z)$ and
$$\omega_{P}^{(n+1)}=\frac{dz}{(z-P)^{n+1}}\;\;\text{for}\;\; P\in k,\quad
\omega_{P}^{(n+1)}=-z^{n-1}dz\;\;\text{for}\;\; P=\infty.$$
Correspondingly,
$$\eta_{P}^{(n)}(z)=-\frac{1}{n(z-P)^{n}}\;\;\text{for}\;\; P\in k,\quad
\eta_{P}^{(n)}(z)=-\frac{z^{n}}{n}\;\;\text{for}\;\; P=\infty.$$
\end{remark}
\begin{remark} \label{serre} Put $\mathbb{O}_{X}=\A_{X}(0)=\prod_{P\in X}\mathcal{O}_{X}$. By Lemma \ref{reciprocity} we have 
$$\A_{X}=\mathcal{A}(X,D_{\mathrm{ns}}) + \mathbb{O}_{X},$$
while Serre's adelic interpretation of cohomology yields
$$\A_{X}/(F+\mathbb{O}_{X})=H^{1}(X,\mathcal{O}_{X}).$$
\end{remark}
\begin{remark} \label{closed} The condition that the constant field $k$ is algebraically closed is not necessary: all results in this section remain valid for any field of constants of characteristic $0$ if we replace the field $k$ by the residue
class field $k(P)=\mathcal{O}_{\pp}/\pp$ and use the trace map $\Tr_{k(P)/k}: k(P)\rightarrow k$. For example, for the bilinear form $(~,~)_{X}$ we have
\begin{equation*}
(\omega_1,\omega_2)_X=\sum_{P\in X}\Tr_{k(P)/k}
\Res_P(d^{-1}\omega_1\,\omega_2).
\end{equation*}
\end{remark}

\begin{remark}
The multiplicative analogue of a $k$-vector space $\mathcal{A}(X)$ of additive multi-valued functions is the group $\mathcal{M}(X)$ of multiplicative multi-valued functions on $X$.  This subgroup of the group of invertible elements of the ad\`{e}le ring $\A_{X}$ is defined by the following properties. It contains $F^{\ast}$ as a subgroup, we have $\displaystyle{d\log m=m^{-1}dm=\omega\in\Omega^1_{F/k}}$ for all $m \in\mathcal{M}(X)$ and if $m\in\mathcal{M}(X)$ satisfies $d\log m=0$, then $m=c\in k^{\ast}$. It also seems natural to assume (as was done in a preliminary version of this paper) that the following multiplicative analogue of Lemma \ref{reciprocity} holds. Every rational function $f\in F^{\ast}$ can be written uniquely as a product of multiplicative multi-valued functions with one zero and one pole obeying the natural generalized
Weil reciprocity law on $X$ (see \cite{Weil},\cite{Serre}). However, the referee pointed out that this assertion contradicts the non-triviality of Poincar\'{e} bi-extension over the square of the Jacobian of $X$ \cite{Gorchinski}.
\end{remark}
\section{Local Theory\label{Local}}
Let $K$ be a complete closed field, that is, a complete discrete valuation field with valuation ring $\mathcal{O}_K$, maximal ideal $\pp$ and
algebraically closed residue field $k=\mathcal{O}_K/\pp$. 
Every local uniformizer $t$ determines an isomorphism
$K\simeq k((t))$. Therefore $K$ may be interpreted as a  `geometric loop algebra' over $k$. The main example of a complete closed field is $K=F_P$, where $P$ is a point on an algebraic curve $X$ over $k$. 

Here we describe some infinite-dimensional Lie algebras naturally associated with $K$ and construct their irreducible highest-weight modules. When $K=F_{P}$, these objects determine local quantum field theories at $P\in X$. Specifically, we consider the following local quantum field theories (QFT):
\begin{enumerate}
\item[\textbf{1.}] the `QFT of additive bosons', which corresponds to the Heisenberg
Lie algebra $\mathfrak{g}$ (a one-dimensional central extension of the geometric loop algebra $\mathfrak{g}\mathfrak{l}_1(K)=K$),
\item[\textbf{2.}] the `QFT of lattice bosons', which corresponds to the lattice
Lie algebra $\mathfrak{l}$ associated with the Heisenberg Lie algebra
$\mathfrak{g}$ and the lattice $\ZZ$.
\end{enumerate}

\subsection{The Heisenberg algebra\label{Heisenberg-Lie-local}} 
Let $\Omega^{1}_{K/k}$ be the $K$-module of K\"{a}hler differentials. We put $\displaystyle{\tilde{\Omega}^{1}_{K/k}=\Omega^{1}_{K/k}/\mathcal{Q}}$, where $\displaystyle{\mathcal{Q}=\cap_{n\geq 0}\,\pp^n d(\mathcal{O})}$ (see Section \ref{differentials}). The abelian Lie algebra $\g\mathfrak{l}_1(K)=K$ over the field $k$ is endowed
with a natural bilinear skew-symmetric form $c: \wedge^2 K\rightarrow k$ by the formula
\begin{equation*}
c(f,g) = -\Res (fdg),\quad f,g\in K,
\end{equation*}
where $dg\in \tilde{\Omega}^{1}_{K/k}$.
The bilinear form $c$ is continuous with respect to the $\mathfrak{p}$-adic topology on $K$
and the discrete topology on $k$. Hence $c\in H^{2}_{\mathrm{c}}(K,
k)\simeq\Hom_{\mathrm{c}}(\wedge^2 K, k)$, where $\Hom_{\mathrm{c}}(\wedge^2 K, k)$ is the group of continuous $2$-cocycles on $K$
with values in $k$.
\begin{definition} The Heisenberg Lie algebra $\g$
is the one-dimensional central extension of $K$
\begin{equation*}
0\rightarrow k\cdot C\rightarrow\g\rightarrow K\rightarrow 0
\end{equation*}
with the $2$-cocycle $c$.
\end{definition}
Writing $[~,~]$ for the Lie bracket in $\g=K\oplus k\cdot C$, we have
\begin{equation*}
[f+a\,C, g+b\,C]=c(f,g)\,C,\quad f,g\in K,~a,b\in k.
\end{equation*}
The Lie subalgebra $\g_+ =\OO_K \oplus k\cdot C$ is a maximal abelian subalgebra of $\g$.
\begin{remark} Let $\Aut\OO =\{u\in\OO : v(u)=1\}$ be the group of continuous automorphisms of the valuation ring $\OO=k[[t]]$
(see~\cite{ben-zvi-frenkel}).
One can easily show that every continuous
linear map $l: k((t)) \otimes_k k((t)) \rightarrow k$ which satisfies
\begin{equation*}
l(f\circ u, g\circ u) = l(f,g)
\end{equation*}
for all $f,g\in k((t))$ and $u\in\Aut\OO$ is a constant multiple of $c$. This explains the natural role  
of the $2$-cocycle $c$ of $K$.
In particular, every $\Aut\OO$-invariant bilinear form
$l$ is necessarily skew-symmetric. This may be regarded as a simple algebraic analogue of  the spin-statistics theorem.
\end{remark}
\begin{definition}
A \emph{module of the Heisenberg algebra} $\g$ is a $k$-vector $V$ with the discrete topology and with a
$k$-algebra homomorphism $\rho:\g\rightarrow\End V$ such that the $\g$-action on $V$
is continuous and $\rho(C)=\bm{I}$ is the identity endomorphism of $V$.
\end{definition}
Equivalently, for every $v\in V$ there is an open
subspace $U$ of $K$ which is commensurable with $\pp$ and annihilates $v$: $\rho(U)\,v=0$.
Putting $\bm{f}=\rho(f)\in\End V$ for all $f\in K$, we have
\begin{equation*}
[\bm{f},\bm{g}]=c(f,g)\bm{I}
\end{equation*}
and thus obtain a projective representation of the abelian Lie algebra $K$.
\begin{remark}
Any choice of the uniformizer $t$ for $K$ determines an isomorphism $K\simeq k((t))$ and a basis
basis $\{t^n\}_{n\in\ZZ}$ in $K$. Putting $\bm{\alpha}_n=\rho(t^n)$ and
using the formulae $c(t^m,t^n)=m\delta_{m,-n}$, we get the commutation relations of the `oscillator
algebra'
\begin{equation*}
[\bm{\alpha}_m,\bm{\alpha}_n]=m\delta_{m,-n}\bm{I}.
\end{equation*}
They characterize free bosons in the two-dimensional QFT.
\end{remark}
\begin{definition}
An \emph{irreducible highest-weight module of the Heisenberg algebra} $\g$ is an
irreducible $\g$-module with a vector $\bm{1}\in V$ which is annihilated by the
abelian subalgebra $\OO_K\oplus \{0\}$.
\end{definition}
The following result is well-known (see, for example, \cite[Lemma 9.13]{kac}).
\begin{theorem}
Each irreducible highest-weight module of the Heisenberg Lie algebra $\g$
is either the trivial one-dimensional module $\mathit{k}=k\cdot\bm{1}$  with the highest vector $\bm{1}=1\in k$, or the Fock module
\begin{equation*}
\cF=\Ind_{\g_{+}}^{\g}k
\end{equation*}
induced from the one-dimensional $\g_+$-module $k$.
\end{theorem}
\begin{remark}
Let $U\g$ be the universal enveloping algebra of the Lie algebra $\g$.
By definition,
$$ \cF 
=U\g \underset{U\g_{+}}{\otimes}k,$$
where $U\g$ is regarded as a right $U\g_{+}$-module. 
Equivalently,
\begin{equation} \label{fock-1}
\cF=\cW/\cD,
\end{equation}
where
$\cW$ is the Weyl algebra  of $\g$, that is, the quotient of $U \g$ by the ideal generated by $C-\bm{1}$ (with now $\bm{1}$ standing for the identity in $U\g$) and $\cD$ is the left ideal in $\cW$ generated by $\mathcal{O}_K\oplus\{0\}$.
\end{remark}

Explicit realization of the Fock module $\cF$ (the bosonic Fock space) depends on a decomposition of $K$ into a
direct sum of subspaces isotropic with respect to the bilinear form $c$:
\begin{equation} \label{decomposition}
K=K_{+}\oplus K_{-},
\end{equation}
where the subspace $K_{+}=\mathcal{O}_K$ is defined canonically. In this case,
\begin{equation} \label{fock-2}
\cF\simeq \Symm^{\bullet} K_{-}
\end{equation}
is the symmetric algebra of the $k$-vector space $K_{-}$.  The Fock space $\cF$ is a $\ZZ$-graded commutative algebra
\begin{equation*}
\cF=
\bigoplus_{n=0}^{\infty}\, \cF^{(n)}
\end{equation*}
where $\cF^{(n)}\simeq \Symm^n K_{-},~\cF^{(0)}=
k\cdot\bm{1}$, and $\cF^{(n)}=\{0\}$ for $n<0$.
For every $f=f_{+} + f_{-} \in K$ the operator $\bm{f}=\rho(f)\in
\End\cF$ is defined by the formula
\begin{equation} \label{action-ab}
\bm{f}\cdot v=f_{-}\odot v+\sum_{i=1}^kc(f,v_i)\,v^i=f_{-}\odot v-\sum_{i=1}^k \Res\,(f_{+}\,dv_i)\,v^i,
\end{equation}
where $v=v_1\odot\cdots\odot v_k\in\cF^{(k)}$ and $
v^i=v_1\odot\cdots\odot\hat{v_i}\odot\cdots\odot v_k\in\cF^{(k-1)},~
i=1,\ldots,k$. Here $\odot$ stands for the multiplication in $\Symm^{\bullet}K_{-}$. In particular,
$$\bm{f}\cdot\bm{1}=f_{-}.$$

The Fock module $\cF$ is endowed with 
the linear topology given by the filtration associated with the $\ZZ$-grading and independent of the decomposition \eqref{decomposition}.

\begin{remark} Any choice of the uniformizer $t$ determines an isomorphism
$K\simeq k((t))$, and one can take $K_{-}=t^{-1}k[t^{-1}]$. The map
\begin{equation*}
\cF^{(n)}\ni v=t^{-m_1}\odot\dots \odot t^{-m_n}\mapsto x_{m_1}\dots
x_{m_n}\in k[x_1, x_2,\dots]
\end{equation*}
determines an isomorphism $\cF\simeq k[x_1, x_2, \dots]$ between the
bosonic Fock space and the polynomial ring in infinitely many variables
$\{x_n\}_{n\in\mathbb{N}}$. Under this map we have $\bm{\alpha}_n\mapsto
n\partial/\partial x_n,~\bm{\alpha}_{-n} \mapsto x_n,~n>0$ (the operator of multiplication
by $x_n$), and $\bm{\alpha}_0\mapsto 0$.
\end{remark}

\begin{remark} \label{splitting}
For an arbitrary complete closed field $K$ there is no canonical choice of the isotropic subspace $K_-$ complementary to $K_+=\OO_K$. 
However, any choice of an  effective non-special divisor $D=P_{1}+\dots +P_{g}$ of degree $g$ on an algebraic curve $X$ and uniformizers $t_{i}$ at $P_{i}$, determines such isotropic subspaces $K_{-}$ for all fields $K=F_{P}$, $P\in X$. Namely, let $\mathcal{A}(X,D)$ be the $k$-vector
space of additive functions defined in Example \ref{Additive}, and let $\mathcal{A}_P(X,D)$ be the subspace of additive functions with the only pole at $P$. We put
\begin{equation*}
K_- =\left.\mathcal{A}_P(X,D)\right|_P \subset K.
\end{equation*}
By part part (i) of Lemma \ref{reciprocity}, the subspace $K_-$ is isotropic with respect to $c$ and we have the decomposition \eqref{decomposition}. The subspace $K_-$ is spanned by the elements $v^{(n)}_P=\left.\eta^{(n)}_P\right|_P,\,n\in\mathbb{N}$, and  $d K_- = \left.\Omega_0(\ast P)\right|_P$.
\end{remark}
The bilinear form $c$ has the one-dimensional kernel $k$. Since
$\mathcal{O}_K/k = \pp$, the form $c$ determines a continuous non-degenerate pairing $c:\pp\otimes
K_{-}\rightarrow k$, whence  $\pp=K^{\vee}_{-}=\Hom(K_{-},k)$
is the topological dual to the $k$-vector space $K_{-}$. 
The topological dual of the bosonic Fock space $\cF$ is accordingly equal to the $k$-vector space
$\cF^{\vee}=\overline{\Symm^{\bullet} \pp}$ which is the completion of $\Symm^{\bullet} \pp$
with respect to the linear topology given by the filtration $\{F^n \Symm^{\bullet},
\pp\}_{n=0}^{\infty}$,
\begin{equation*}
F^n \Symm^{\bullet} \pp =\oplus_{i=0}^n\Symm^i\pp.
\end{equation*}
The continuous pairing $(~,~):\cF^{\vee}\otimes\cF\rightarrow k$ is uniquely determined by the pairing between $\Symm^{\bullet} \pp$ and 
$\cF=\Symm^{\bullet} K_{-}$ and is defined recursively by the formula
\begin{equation} \label{dual}
(u,v)=\delta_{kl}\sum_{i=1}^{l}c(u_{1},v_{i})(u^{1},v^{i}),
\end{equation}
where $u=u_1\odot\cdots\odot
u_k=u_{1}\odot u^{1}\in\Symm^k\pp$ and $v=v_1\odot\cdots \odot v_l=v_{i}\odot v^{i}\in\cF^{(l)}$.
The dual bosonic Fock space $\cF^{\vee}$ is a right $\g$-module with lowest-weight vector $\bm{1}^{\vee}$ annihilated by the subspace $K_{-}\oplus k$.

The representation $\rho$ of the Heisenberg algebra $\g$ on $\cF$ determines a
contragradient representation $\rho^{\vee}$ of $\g$ the Heisenberg algebra on $\cF^{\vee}$ by the formula
\begin{equation*} 
(u\cdot \rho^{\vee}(f),v)=(u,\rho(f)\cdot v),~\forall\, u\in\cF^{\vee},\,v\in\cF. 
\end{equation*}
More explicitly, put $f=\tilde{f}_{+}+\tilde{f}_{-}\in K$, where now $\tilde{f}_{+}\in \pp$ and $\tilde{f}_{-}\in K_{-}\oplus k$. Then it follows from \eqref{action-ab} and \eqref{dual} that the operator $\bm{f}=\rho^\vee(f)\in\End\cF^{\vee}$ is given by 
\begin{equation} \label{action-ab-dual}
u\cdot\bm{f}=\tilde{f}_{+}\odot u+ \sum_{i=1}^{k}c(u_i,f)u^i =\tilde{f}_{+}\odot u+\sum_{i=1}^{k}\Res\,(\tilde{f}_{-}du_i)u^i,
\end{equation} where
$u=u_1\odot\cdots\odot u_k\in\Symm^{k}\pp$ and
$u^i=u_1\odot\cdots\odot\hat{u_i}\odot\cdots\odot u_k\in\Symm^{k-1}\pp$.
\subsection{The lattice algebra\label{lattice algebra-local}} 
Let $k[\ZZ]$ be the group algebra of the additive group $\ZZ$.
As a $k$-vector space, $k[\ZZ]$ has a basis $\{e_n\}_{n\in\ZZ}$,
$e_{m}e_{n}=e_{m+ n}$. For every decomposition \eqref{decomposition} we define the `constant term' of any $f\in K$ as
$f(0)=f_{+}\!\!\!\mod\mathfrak{p}\in k$. Hence we have $f(0)=0$ for $f\in K_{-}$.
\begin{remark} If $K=F_{P}$ and $K_{-}=\left.\mathcal{A}_{P}(X,D)\right|_{P}$, then $f(0)$ is the constant term of the formal Laurent expansion of $f\in k((t_{P}))$ with respect to the uniformizer $t_{P}$ for $K$, defined in Section \ref{A-functions}. 
\end{remark}
\begin{definition} The lattice algebra $\mathfrak{l}$ associated with the decomposition \eqref{decomposition} is a semidirect sum of the Heisenberg Lie algebra $\mathfrak{g}$ and the abelian Lie algebra $k[\ZZ]$ with the Lie bracket
\begin{equation*}
[f+ aC +\alpha e_{m}, g+bC +\beta e_{n}]=c(f,g)C +m\alpha g(0)e_{m}-n\beta f(0)e_{n},
\end{equation*}
where $f+ aC, g+bC\in\mathfrak{g}$.
\end{definition}

The corresponding irreducible highest-weight module $\cB$ for the lattice algebra $\mathfrak{l}$ is given by
\begin{equation*}
\cB=k[\ZZ]\otimes\cF,
\end{equation*}
where $k[\ZZ]$ acts by multiplication and $K$ acts by the formula
$$\bm{f}(e_{n}\otimes v)=-nf(0)e_{n}\otimes v +e_{n}\otimes (\bm{f} \cdot v),\quad f\in K,\; v\in\cF.$$

The module $\cB$ (the Fock space of  `charged bosons') is a $\ZZ$-graded commutative algebra,
$$\cB=\bigoplus_{n\in\ZZ}\cB^{(n)},\quad \cB^{(n)}=k\cdot e_{n}\otimes \cF.$$
The elements $e_{n}$, $n\in\ZZ$, correspond to the shift operators $\bm{e}_{n}=\bm{e}^{n}$ in
$\cB$, where 
$$\bm{e}(e_{n}\otimes v)=e_{n+1}\otimes v,\quad v\in\cF.$$
\begin{remark}  \label{Z} Using the canonical isomorphism $K^{\ast}/\mathcal{O}^{\ast}_{K}\simeq \ZZ$ induced by the valuation $v: K^{\ast}\rightarrow\ZZ$, one can also define the Fock space $\cB$ as the space of all functions
$$F: K^{\ast}/\mathcal{O}^{\ast}_{K}\rightarrow\cF$$
with finite support. 
\end{remark}
\begin{remark} For
every choice of the uniformizer $t$ for $K$, the map
$$\cB^{(n)}\ni e_{n}\otimes (t^{-m_{1}}\odot\cdots\odot t^{-m_{l}})\mapsto e^{nx_{0}}x_{m_{1}}\dots x_{m_{l}}\in e^{nx_{0}}k[x_{1},x_{2},\dots]$$
establishes the isomorphism
$\cB\simeq k[e^{x_{0}},e^{-x_{0}},x_{1},x_{2},\dots]$.
Under this map we have $\bm{\alpha}_n\mapsto
n\partial/\partial x_n,~\bm{\alpha}_{-n} \mapsto x_n,~n>0$, 
$\bm{\alpha}_0\mapsto -\partial/\partial x_{0}$, and $\bm{e}\mapsto e^{x_{0}}$ (the operator of multiplication by $e^{x_{0}}$).
\end{remark}
The topological dual of $\cB$ is the $k$-vector space
$$\cB^{\vee}=\bigoplus_{n\in\ZZ}\,\,k\!\cdot\! q^{n}\otimes\cF^{\vee},$$ 
where $\{q^{n}\}_{n\in\ZZ}$ is the basis in
$k[\ZZ]^{\vee}$ dual to the basis $\{e_{n}\}_{n\in\ZZ}$.
The continuous pairing $(~,~):\cB^{\vee}\otimes\cB\rightarrow k$ is
given by
\begin{displaymath}
(q^{m}\otimes u,e_{n}\otimes v)=(u,v)\delta_{mn},\quad u\in\cF^{\vee}, \;v\in\cF.
\end{displaymath}

As for the Heisenberg algebra, 
the representation $\rho$ of the lattice Lie algebra $\mathfrak{l}$ on $\cB$ determines a contragradient representation $\rho^{\vee}$ on $\cB^{\vee}$. The dual Fock space $\cB^{\vee}$ is a right $\mathfrak{l}$-module with
lowest-weight vector $\bm{1}^{\vee}$ annihilated by $K_{-}$.

\section{Global Theory\label{Global}}
Given an algebraic curve $X$ over an algebraically closed field $k$ of characteristic 0,
we shall define the global versions of the local QFT's introduced in the previous section. One can briefly characterize these global QFT's as follows.
\begin{enumerate}
\item[1.] The `QFT of additive bosons on $X$' corresponds
to the global Heisenberg algebra $\mathfrak{g}_X$ (the restricted direct sum of local Heisenberg algebras $\mathfrak{g}_P$  over all points $P\in X$). The
global Fock space $\cF_X$ is defined as the restricted tensor product
of the local Fock spaces $\cF_P$ over all points $P\in X$. The 
global Fock space $\cF_X$ is a highest-weight $\mathfrak{g}_X$-module. There is a linear functional
$\langle\,\cdot\,\rangle\,:\cF_X\rightarrow k$ (the expectation value
functional) which is uniquely determined by the properties of normalization and invariance with respect to
the space of additive functions.
\item[2.] The `QFT of charged bosons on $X$' corresponds to the 
global lattice algebra $\mathfrak{l}_X$. The global charged Fock space $\cB_X$ is a highest-weight $\mathfrak{l}_X$-module and there is a unique expectation value functional
$\langle\,\cdot\,\rangle: \cB_X\rightarrow k$ with similar properties.
\end{enumerate}
\subsection{Additive bosons on $X$\label{AB}}
The QFT of additive bosons consists of the following data.
\begin{enumerate}
\item[AB1)] An effective non-special divisor $D_{\mathrm{ns}}=P_1+\dots + P_g$ of degree $g$ on $X$ with distinct points, uniformizers
$t_i$ at $P_i$ and the $k$-vector space of additive functions
$\mathcal{A}(X,D_{\mathrm{ns}})$ (a subspace of $\A_X$ containing $F=k(X)$) introduced in Example \ref{Additive}.
\item[AB2)] The local QFT's of additive bosons (the highest-weight
$\g_P$-modules $\cF_P$ for all points $P\in X$).
\item[AB3)] The global Heisenberg algebra $\g_X$ (the one-dimensional
central extension of the abelian Lie algebra
$\g\mathfrak{l}_1(\A_X)=\A_X$ by the cocycle $c_X=\sum_{P\in X} c_P$).
\item[AB4)] A highest-weight $\g_X$-module $\cF_X$ (the global Fock space, which is the restricted tensor product of $\cF_P$ over
all points $P\in X$).
\item[AB5)] An expectation value functional, that is, a
linear map $\langle\,\cdot\,\rangle: \cF_X\rightarrow k$ with the following properties:
\begin{itemize}
\item[(i)] $\langle\textbf{1}_{X} \rangle=1$, where $\textbf{1}_{X}\in\cF_X$
is the highest-weight vector,
\item[(ii)] $\langle \bm{a}\cdot v\rangle=0$ for all $a\in\mathcal{A}(X,D_{\mathrm{ns}})$ and
$v\in\cF_X$.
\end{itemize}
\end{enumerate}

The data AB1) and AB2) have already been described in Sections  \ref{second kind} and
\ref{Heisenberg-Lie-local}. Here we introduce the global Heisenberg algebra $\g_X$, construct the corresponding global Fock space $\cF_X$ and prove that there is a unique expectation value functional  $\langle\,\cdot\,\rangle$ with properties (i) and (ii).

Let
$c_X:\A_X \times \A_X\rightarrow k$ be the global bilinear form
\begin{equation*}
c_X(x,y)=\sum_{P\in X}c_P(x_P,y_P) = -\sum_{P\in X}\Res_P(x_Pd y_P),\quad x,y\in\A_X.
\end{equation*}
\begin{definition} The \emph{global Heisenberg Lie algebra} $\g_X$
is the one-dimensional central extension of the abelian Lie algebra
$\A_X$
\begin{equation*}
0\rightarrow k\,C\rightarrow\g_X\rightarrow \A_X\rightarrow 0
\end{equation*}
by the two-cocycle $c_X$.
\end{definition}

The Lie subalgebra $\g_{X}^{+}=\mathbb{O}_X\oplus kC$ is the maximal abelian subalgebra of $\g_{X}$.
\begin{definition}
The \emph{global Fock space} $\cF_X$ is an irreducible $\g_{X}$-module with vector $\bm{1}_{X}$ annihilated 
by the abelian
subalgebra $\mathbb{O}_X\oplus \{0\}$.
\end{definition}
As in the local case, the global Fock module is induced from the one-dimensional $\g_{X}^{+}$--module: 
$$\cF_{X}=\Ind_{\g_{X}^{+}}^{\g_{X}} k.$$
By what was said in the previous section,  we have a decomposition \eqref{decomposition} for $K=F_{P}$, $P\in X$, where $F_{P}^{(+)}=\mathcal{O}_{P}$ and $F_{P}^{(-)}=\left.\mathcal{A}_{P}(X,D)\right|_{P}$.
This yields the following decomposition 
of the $k$-vector space $\A_X$ into a direct sum of subspaces isotropic with respect to $c_X$:
\begin{equation} \label{decomposition-global}
\A_X = \mathbb{O}_{X}\oplus\mathcal{F}_{X}^{(-)}.
\end{equation}
Here
\begin{equation*}
\mathcal{F}_X^{(-)} =\coprod_{P\in X} F_P^{(-)}
\end{equation*}
is the restricted direct product over all $P\in X$ with respect to the zero subspaces $\{0\}\subset F_{P}^{(-)}$.
The decomposition \eqref{decomposition-global} gives rise to an isomorphism
\begin{equation*}
\cF_X\simeq\Symm^\bullet\mathcal{F}_X^{(-)}.
\end{equation*}
The global Fock space $\cF_X$ carries a linear topology given by the natural
filtration associated with the $\ZZ$-grading.

Equivalently,  $\cF_X$ may be defined as the tensor product
\begin{equation*}
\cF_X = \underset{P\in X}{\widehat{\otimes}}\cF_P,
\end{equation*}
which is restricted with respect to the vectors $\bm{1}_P\in\cF_P$ and is endowed with the
product topology. In other words, $\bm{1}_X =\otimes_{P\in X} \bm{1}_P$, and $\cF_X$ is spanned by the vectors
$$v= \underset{P\in X}\otimes v_P,$$ where $v_P=\bm{1}_P$ for all but finitely many $P\in X$.
For every $P\in X$ we have $v=v_{P}\otimes v^{P}$, where $v^{P}=\otimes_{Q\in X}\tilde{v}_{Q}$,  $\tilde{v}_{Q}=v_{Q}$ for $Q\neq P$ and $\tilde{v}_{P}=\bm{1}_{P}$.
We denote the corresponding representation of $\g_{P}$ on $\cF_{P}$,  $P\in X$, by
$\rho_{P}$, and the representation of $\g_{X}$ on $\cF_{X}$ by $\rho$. Putting $\bm{x}=\rho(x)\in\End\cF_{X}$ for $x=\{x_{P}\}_{P\in X}\in\A_{X}$ and taking any $v=\otimes_{P\in X} v_P$, we have
$$\bm{x}\cdot v=\sum_{P\in X} \bm{x}_{P}\cdot v_P\otimes v^{P},$$
where $\bm{x}_{P}=\rho_{P}(x_{P})\in\End\cF_{P}$.

Put
\begin{equation*}
\mathfrak{P}_X =\prod_{P\in X}\pp.
\end{equation*}
The topological dual of the global Fock space $\cF_X$ is the $k$-vector space
$\cF_X^{\vee}=\overline{\Symm^{\bullet}\mathfrak{P}_X}$, which is the completion of
$\Symm^{\bullet}\mathfrak{P}_X$ with respect to the linear topology given by the natural
filtration associated with the $\ZZ$-grading. The dual global Fock space
$\cF_X^{\vee}$ is a right $\g_X$-module with lowest-weight vector $\bm{1}_X^{\vee}$ annihilated by the abelian subalgebra $\mathcal{F}_X^{(-)}\oplus\{0\}$. Equivalently,
\begin{equation*}
\cF_X^{\vee} = \overline{\underset{P\in
X}{\widehat{\otimes}}\cF_P^\vee}
\end{equation*}
is the completion of the tensor product restricted with respect to the
vectors $\bm{1}_P^{\vee}$. The completion is taken with respect to the double
filtration $\{F^{mn}\Symm^\bullet\mathfrak{P}_X\}$,
\begin{equation*}
F^{mn}\Symm^\bullet\mathfrak{P}_X= \sum_{i=0}^m\sum_{P_1,\dots,P_i\in X}
\left(\bigoplus_{l_1+\dots + l_i=0}^n\Symm^{l_1}
\pp_1\otimes\dots\otimes\Symm^{l_i}\pp_i\right).
\end{equation*}
In other words, the elements of $\cF_X^\vee$ are infinite sums
\begin{equation*}
u=\sum_{n=0}^\infty\sum_{P_1,\dots,P_n\in X}a_{\scriptscriptstyle{P_1\dots P_n}}u_{\scriptscriptstyle{P_1\dots P_n}},
\end{equation*}                                                                \
where the $u_{\scriptscriptstyle{P_1\dots P_n}}\in \overline{\cF^\vee}_{P_1\dots P_n}$ belong to the
completion of the tensor product
\begin{displaymath}
\cF_{P_1\dots
P_n}^\vee=\cF_{P_1}^\vee\otimes\dots\otimes\cF_{P_n}^\vee
\end{displaymath}
with respect to the filtration
\begin{equation*}
F^{m}\cF_{P_1\dots P_n}^\vee
= 
\bigoplus_{l_1+\dots + l_n=0}^m\left(\Symm^{l_1}\pp_1\otimes\dots\otimes\Symm^{l_n}
\pp_n\right).
\end{equation*}
Let $\{u^{(n)}_P\}_{n\in\NN}$ be the basis of $\pp$ dual to the basis $\left\{v^{(n)}_P=\left.\eta^{(n)}_P\right|_P\right\}_{n\in\NN}$ of $F_P^{(-)}$ with respect to the pairing given by $c_{P}$ (see Section \ref{Heisenberg-Lie-local}).  
Then we see that $\cF_X^\vee$ is the completion of the space $k[[u_{P}^{n}]]$ of formal Taylor series in infinitely
many variables $u^{(n)}_P,\,P\in X, n\in\NN$. This realization of $\cF_X^\vee$ is used 
to prove the following main result in the QFT of additive bosons. 
\begin{theorem} \label{additive theorem}
There is a unique linear functional $\langle\, \cdot\,\rangle:\cF_X\rightarrow k$ (the
expectation value functional) with the following properties:
\begin{itemize}
\item[EV1)] $\langle\bm{1}_X\rangle = 1$,
\item[EV2)] $\langle\bm{a}\cdot v\rangle =0$ for all $a\in\mathcal{A}(X,D_{\mathrm{ns}})$ and
$v\in\cF_X$.
\end{itemize}
The functional  $\langle\, \cdot\,\rangle$ is given by
\begin{displaymath}
\langle v\rangle =\left(\Omega_X,v\right),
\end{displaymath}
where
\begin{equation*}
\Omega_X=\exp\left\{-\frac{1}{2}\sum_{m,n=1}^\infty\sum_{P,Q\in X} c^{(mn)}_{PQ}
u^{(m)}_Pu^{(n)}_Q\right\}\in\cF_X^\vee,
\end{equation*}
\begin{displaymath}
c^{(mn)}_{PQ} = -\Res_Q(\eta^{(m)}_P d\eta^{(n)}_Q).
\end{displaymath}
\end{theorem}
\begin{proof}
It follows from decomposition \eqref{AD} that 
a linear functional of the form $\langle v\rangle =(\Omega,v)$ possesses properties
EV1) and EV2) if and only if it is normalized, $(\Omega,\bm{1}_X)=1$, and $\Omega\in\cF_{X}^{\vee}$ satisfies the system of equations
\begin{equation} \label{omega-eta}
\Omega\cdot\boldsymbol{\eta^{(n)}_P}=0
\end{equation}
 for all $P\in X$ and $n\in\NN$,
where $\boldsymbol{\eta^{(n)}_P}=\rho^{\vee}(\eta^{(n)}_P)$.
Write $\eta^{(n)}_P =\beta^{(n)}_{P} + \gamma^{(n)}_{P}$, where 
$\beta^{(n)}_{P}=\{\beta^{(n)}_{PQ}\}_{Q\in X}$ and $\gamma^{(n)}_{P}=\{\gamma^{(n)}_{PQ}\}_{Q\in X}\in\A_{X}$
are given by
\begin{equation*}
\beta^{(n)}_{PQ} 
=\begin{cases} 0& \text{if $Q=P$}, \\
\left.\eta^{(n)}_P\right|_Q & \text{if $Q\neq P$},
\end{cases} \quad\;
\gamma^{(n)}_{PQ}
=\begin{cases} \left.\eta^{(n)}_P\right|_P& \text{if $Q=P$}, \\
 0& \text{if $Q\neq P$}.
\end{cases}
\end{equation*}
It follows from \eqref{action-ab-dual} that $\bm{\gamma^{(n)}_{P}}$ acts on $\cF_X^\vee$ as
differentiation with respect to $u^{(n)}_P$.
For $Q\neq P$ we have
$$\beta_{PQ}^{(n)}=a^{(n)}_{PQ} + \sum_{m=1}^{\infty}a_{PQ}^{(nm)}u_{Q}^{(m)},$$
where $a^{(n)}_{PQ}\in k$ and
$$a_{PQ}^{(nm)}=c(\beta_{PQ}^{(n)},v_{Q}^{(m)})=-\Res_{Q}(\eta_{P}^{(n)}d\eta_{Q}^{(m)})=c_{PQ}^{(nm)}.$$
Since $c_{PP}^{(nm)}=0$ (see Lemma \ref{reciprocity}), we conclude that $\bm{\beta^{(n)}_{P}}$ acts on 
$\cF_{X}^{\vee}$ as a
multiplication by
$\sum_{Q\in X} c^{(nm)}_{PQ} u^{(m)}_Q$.
One can rewrite the  equations \eqref{omega-eta} in the form
\begin{equation} \label{system}
\left(\frac{\partial}{\partial u^{(n)}_P} + \sum_{Q\in X} c^{(nm)}_{PQ} u^{(m)}_Q\right)\Omega=0, \quad P\in X,\;n\in\NN.
\end{equation}
It follows from part (i) of Lemma \ref{reciprocity} that 
$$c_{PQ}^{(mn)}=c_{QP}^{(nm)},$$
whence the system of differential equations \eqref{system} is compatible and $\Omega_X$ is its unique normalized solution.
\end{proof}
\begin{remark}  \label{currents} Let $\g$ be a semi-simple Lie algebra over $k$ with the Cartan-Killing form $\langle~,~\rangle$. Then the $k$-vector space $\A_{X}$ with bilinear form $c_{X}$ may be replaced  by the $k$-vector space
$\mathbb{V}_{X}=\g\otimes_{k}\A_{X}$ with bilinear form
$- \sum_{P\in X}\Res_P\langle x_P,d y_P\rangle$.
Theorem \ref{additive theorem} extends to this case. The additive Ward identities hold for $\g\otimes_{k}\mathcal{A}(X,D_{\mathrm{ns}})$ and the corresponding QFT is associated with the current algebra on $X$ in the sense of \cite{Witten-2}.
\end{remark}
\subsection{Charged additive bosons on $X$\label{CB}}
The QFT of charged additive bosons is determined by the following data.
\begin{enumerate}
\item[CB1)] An effective non-special divisor $D_{\mathrm{ns}}=P_1+\dots + P_g$ of degree $g$ on $X$ with distinct points, uniformizers 
$t_i$ at $P_i$ and the $k$-vector space of additive functions $\mathcal{A}(X,D_{\mathrm{ns}})$ (a subspace of $\A_X$ containing $F=k(X)$) introduced in Example \ref{Additive}.
\item[CB2)] The local QFT's of charged additive bosons (the highest-weight
$\mathfrak{l}_P$-modules $\cB_P$ for all points $P\in X$).
\item[CB3)] The global lattice algebra $\mathfrak{l}_X$ (a 
semi-direct sum of the global Heisenberg algebra $\g_{X}$ and the abelian Lie algebra $k[\Div_{0}(X)]$ with generators $e_{D}$, $D\in\Div_{0}(X)$, where $k[\Div_{0}(X)]$ is the group algebra of the additive 
group $\Div_{0}(X)$ of degree $0$ divisors on $X$).
\item[CB4)] A highest-weight $\mathfrak{l}_X$-module $\cB_X$ (the global Fock space
 with the highest-weight vector $\bm{1}_{X}\in\cB_X$).
\item[CB5)] An expectation value functional, that is, a 
linear map $\langle\,\cdot\,\rangle: \cB_X\rightarrow k$ with
the following properties:
\begin{itemize}
\item[(i)] $\langle\bm{e}_{\bm{D}}\cdot \bm{1}_{X} \rangle=1$ for
all $D\in\Div_{0}(X)$,
\item[(ii)] $\langle \bm{a}\cdot u\rangle=0$ for all $a\in\mathcal{A}(X,D_{\mathrm{ns}})$
and $u\in\cB_X$.
\end{itemize}
\end{enumerate}

As a $k$-vector space, the group algebra $k[\Div_0(X)]$ has a basis $\{e_D\}_{D\in\Div_0(X)}$,
$e_{D_1}e_{D_2}=e_{D_1 + D_2}$. 
For every $x=\{x_P\}\in\A_X$ and $D=\sum_{P\in X} n_P\,P\in\Div_0(X)$, we put
\begin{displaymath}
x(D) = \sum_{P\in X} n_P x_P(0)\in k,
\end{displaymath}
where $x_P(0)=x_P^{+}\!\!\mod\pp\in k$ is the constant term of $x_P\in F_P$, (it is determined by decomposition \eqref{decomposition} associated with the non-special divisor $D_{\mathrm{ns}}$; see Section \ref{lattice algebra-local}).
\begin{definition}
The \emph{global lattice algebra} $\mathfrak{l}_{X}$ is the semi-direct sum of the global Heisenberg
algebra $\g_X$ and the abelian Lie algebra $k[\Div_0(X)]$ with Lie bracket
\begin{displaymath}
[x + \alpha C+ \gamma e_{D_1}, y+\beta C + \delta e_{D_2}] = c_X(x,y)C +y(D_1)\gamma e_{D_1}- x(D_2)\delta e_{D_2} ,
\end{displaymath}
where $x+\alpha C, y+\beta C\in\g_X$ and $\gamma,\delta\in k$.
\end{definition}
The global Fock space $\cB_X$ is the tensor product of
the group algebra $k[\Div_0(X)]$ and
the Fock space of additive bosons $\cF_X$:
\begin{equation*}
\cB_X =k[\Div_0(X)]\otimes \cF_X  =
\bigoplus_{D\in\Div_0(X)}\cB_X^D,
\end{equation*}
where
\begin{displaymath}
\cB_X^D=k\cdot e_D\otimes\cF_X.
\end{displaymath}
$\cB_X$ is an irreducible $\mathfrak{l}_X$-module where $k[\Div_{0}(X)]$ acts by multiplication:
\begin{align}
\bm{e}_{\bm{D_{1}}}(e_{D_{2}}\otimes v) & =e_{D_{1}+D_{2}}\otimes v,\quad v\in\cF_{X}, \label{A-1} \\
\intertext{and $\A_{X}$ acts by the formula}
\bm{x}(e_{D}\otimes v) & =-x(D)e_{D}\otimes v
+e_{D}\otimes(\bm{x}\cdot v),\quad x\in\A_{X},\;v\in\cF_{X}. \label{A-2}
\end{align}
For every $D =\sum_{P\in X}n_P\,P
\in\Div_0(X)$ the subspace $\cB_X^D$ is an
irreducible $\g_X$-module with the following property. If $x=\{x_P\}_{P\in X}\in\A_X$ with $x_P\in k$ for
all $P\in X$, then the restriction of the operator $\bm{x}$ to $\cB_X^D$ is equal to $-x(D)\bm{I}$, where 
$\bm{I}$ is the identity operator. In particular, when $x=c$ is a constant, we have $x(D)=c\deg D=0$, and $\bm{x}$ acts by zero on $\cB_{X}$. 

\begin{remark} One can also define an extended global lattice algebra
$\tilde{\mathfrak{l}}_{X}$ as a semidirect sum of the global Heisenberg
algebra $\g_X$ and the abelian Lie algebra $k[\Div(X)]$. The corresponding irreducible $\tilde{\mathfrak{l}}_{X}$-module is the extended Fock space
\begin{equation*}
\tilde{\cB}_{X}=k[\Div(X)]\otimes \cF_X =
\bigoplus_{D\in\Div(X)}\cB_X^D.
\end{equation*}
The action of $\tilde{\mathfrak{l}}_{X}$  on $\tilde{\cB}_{X}$ is given by the same formulas \eqref{A-1}--\eqref{A-2}, where now the constant ad\`{e}le $x=c$ acts on $\cB_{X}^{D}$ by $(c\deg D)\bm{I}$.
\end{remark}
The dual Fock space $\cB^\vee_X$ is defined as a
completion of the direct sum of dual spaces to 
$\cB_X^D$ over $D\in\Div_{0}(X)$. This completion is given by formal infinite sums. 
Explicitly, 
\begin{equation*}
\cB^\vee_X =\overline{ \bigoplus_{D\in\Div_0(X)}\cB_X^\vee(D)},
\end{equation*}
where
\begin{displaymath}
\cB_X^\vee(D) =k\cdot q^D\otimes\cF_X^\vee,
\end{displaymath}
$q^D\in k[\Div_{0}(X)]^{\vee}$ are dual to $e_D$, and $\cF_{X}^{\vee}$ was defined in Section \ref{AB}.

\begin{theorem} \label{charged theorem}
There is a unique linear functional $\langle\,\cdot\,\rangle:\cB_X
\rightarrow k$ (the expectation value functional) with the following properties:
\begin{itemize}
\item[EV1)] $\langle\bm{e}_{\bm{D}}\cdot \bm{1}_X\rangle = 1$ for all $D\in\Div_{0}(X)$,
\item[EV2)] $\langle\bm{a}\cdot v\rangle =0$ for all $a\in\mathcal{A}(X,D_{\mathrm{ns}})$ and $v\in\cB_X$.
\end{itemize}
The functional $\langle\,\cdot\,\rangle$ is given by
\begin{displaymath}
\langle v\rangle =(\hat{\Omega}_X,v),
\end{displaymath}
where
\begin{equation*}
\hat{\Omega}_X=\sum_{D\in\Div_0(X)} q^D\otimes
\exp\left\{\sum_{n=1}^\infty\sum_{P\in X}\eta^{(n)}_{P}(D)
u^{(n)}_P\right\}\Omega_X \in\cB_X^\vee,
\end{equation*}
and $\Omega_{X}$ is defined in Theorem \rm{\ref{additive theorem}}.
\end{theorem}
\begin{proof} We put
\begin{displaymath}
\Omega=\sum_{D\in\Div_0(X)} q^D\otimes \Omega_D,\quad\Omega_D\in\cF_X^\vee.
\end{displaymath}
The condition $(\Omega, e_{D}\otimes\bm{1}_{X})=1$ for all $D\in\Div_0(X)$ 
is equivalent to the normalization $(\Omega_{D}, \bm{1}_{X})=1$. 
Since the constants act by zero on $\cB_{X}$, it suffices to verify that  
\begin{equation} \label{vanish}
(q^D\otimes \Omega_D)\cdot\bm{\eta^{(n)}_P}=0
\end{equation}
for all $D=\sum_{Q\in X}n_Q\,Q\in\Div_0(X)$ and $P\in X$. Since
\begin{equation*}
q^D\cdot\bm{\eta^{(n)}_P} =-\eta^{(n)}_{P}(D)\,q^{D}=-\sum_{Q\in X}n_Q\left.\eta^{(n)}_P\right|_Q(0) 
\,q^{D}
\end{equation*}
(note that $\left.\eta^{(n)}_P\right|_P(0)=0$ by the definition in Section \ref{lattice algebra-local}), 
we see from \eqref{vanish} that $\Omega_D$ satisfies the following system of differential equations
\begin{equation*}
\left(\frac{\partial}{\partial u^{(n)}_P} - \sum_{Q\in X}n_Q\left.\eta^{(n)}_P\right|_Q(0)
+\sum_{Q\in X} c^{(nm)}_{PQ} u^{(m)}_Q \right)\Omega_D =0.
\end{equation*}
This system has a unique normalized solution given by
\begin{equation*}
\Omega_{D}=\exp\left\{\sum_{n=1}^\infty\sum_{P\in X}\eta^{(n)}_{P}(D)
u^{(n)}_P-\frac{1}{2}\sum_{m,n=1}^\infty\sum_{P,Q\in X} c^{(mn)}_{PQ}
u^{(m)}_P u^{(n)}_Q\right\}.
\qedhere
\end{equation*} 
\end{proof}
\begin{remark} Theorems \ref{additive theorem} and \ref{charged theorem} hold for an arbitrary field $k$ of constants of characteristic $0$ (see Remark \ref{closed}).
\end{remark}

\begin{remark} All results in this section hold trivially in the case when 
$X$ has genus 0. Using Remark \ref{0-a}, one can easily obtain elementary explicit formulae for the expectation value functional $\langle\,\cdot\,\rangle$ for quantum additive and charged bosons on $\mathbb{P}^{1}_{k}$. 
\end{remark}
\subsection{Invariant formulation\label{Inv}}
Here we present an invariant formulation and
a proof of a generalization of Theorem  \ref{additive theorem} for the current algebra. They were suggested
by the referee. Let $V$ be a $k$-vector space regarded as abelian Lie algebra over $k$, and let $c$ be a skew-symmetric bilinear form on $V$. We write $\tilde{V}$ for the one-dimensional central extension of $V$ 
\begin{equation*}
0\rightarrow k\cdot C\rightarrow\tilde{V}\rightarrow V\rightarrow 0
\end{equation*}
with the 2-cocycle $c$, 
and $\cW$ for the Weyl algebra of the Lie algebra $\tilde{V}$, as in Section \ref{Heisenberg-Lie-local}. Let $U$ and $W$ be isotropic subspaces of $V$ with respect to $c$ such that $U\cap W$ and $V/(U+W)$
are finite-dimensional and $U\cap W$ lies in the kernel of $c$. 
\begin{lemma} There is a canonical isomorphism of $k$-vector spaces
$$\cW/\cW\!\cdot\!(U+W)\simeq \Symm^{\bullet}\left(V/(U+W)\right).$$
\end{lemma}
\begin{proof}
This is proved by direct calculation in a symplectic basis of $V$ compatible with the corresponding bases in $U$ and $W$.
\end{proof}
In the notation of Remark \ref{currents} we put $V=\mathbb{V}_{X}=\g\otimes_{k}\A_{X}$, 
$$c(x,y)= - \sum_{P\in X}\Res_P\langle x_P,d y_P\rangle,$$
and $U=\g\otimes_{k} F$, $W=\g\otimes_{k}\mathbb{O}_{X}$, where $F=k(X)$. Then 
$$\cW/\cW\!\cdot\!W\simeq\cF_{X}$$ 
is the Fock space of the current algebra on $X$. Using Serre's adelic interpretation of cohomology in the form
$$V(U+W)\simeq \g\otimes_{k}H^{1}(X,\mathcal{O}_{X}),$$
we obtain from Lemma 5.1 that
$$\cF_{X}/(\g\otimes_{k} F)\!\cdot\!\cF_{X}\simeq\Symm^{\bullet}(\g\otimes_{k}H^{1}(X,\mathcal{O}_{X})).$$
This shows that the global symmetries $\g\otimes_{k} F$ do not uniquely determine the expectation value functional
$\langle\,\cdot\,\rangle$ except in the case when $X=\mathbb{P}^{1}_{k}$. 

To extend the Lie algebra of global symmetries, we consider a Lagrangian subspace 
$L\subset H_{\mathrm{dR}}^{1}(X)\simeq \Omega^{(2\mathrm{nd})}/dF$ such that the restriction to $L$
of the natural map $H_{\mathrm{dR}}^{1}(X)\rightarrow H^{1}(X,\mathcal{O}_{X})$ 
 is an isomorphism. For example, take $L=k\cdot\omega_{1}\oplus\cdots\oplus k\cdot\omega_{g}$ (see Theorem \ref{Chevalley}). Let
$\tilde{L}$ be the inverse image of $L$ under the map $\Omega^{(2\mathrm{nd})}\rightarrow \Omega^{(2\mathrm{nd})}/dF$.
We claim that there is a subspace $U_{0}\subset \A_{X}$ such that 
$$F\subset U_{0},\quad U_{0}\cap\mathbb{O}_{X}=k\quad\text{and}\quad dU_{0}=\tilde{L}.$$
For example, take $U_{0}=\mathcal{A}(X,D_{\mathrm{ns}})$. Then $U=\g \otimes_{k}U_{0}$ is an isotropic subspace of $V$ and, by  Remark \ref{serre}, we have 
$$V/(U+W)=\{0\}.$$ 
Therefore,
$$\cF_{X}/U\cdot \cF_{X}\simeq k,$$
which is essentially Theorem \ref{additive theorem} (without an explicit formula for the vector $\Omega_{X}$). By Remark \ref{closed}, the condition that the field $k$ is algebraically closed is not necessary.
\bibliographystyle{plain}
\bibliography{qft}
\end{document}